\documentclass[11pt]{article}
\usepackage{fullpage} % 1 inch margin; [cm] for 1.5 cm
\usepackage[indent]{parskip} % Zero \parindent and nonzero \parskip; [indent] for nonzero \parindent

\usepackage[T1]{fontenc} % To copy-and-paste accented characters, etc.
\usepackage{lmodern} % Additional accented characters, etc.
\usepackage{newtxtext}
\makeatletter%
\@ifpackageloaded{xeCJK}{}{%
    \usepackage[utf8]{inputenc}}
\makeatother

\usepackage[shortlabels,inline]{enumitem}

\usepackage{amsmath,amssymb,amsfonts}
\usepackage{dsfont}
\usepackage{bm}
\usepackage[mathscr]{eucal} % To change \mathcal. Use [mathscr] to add \mathscr instead
\usepackage{amsthm}
\usepackage[mleftright]{diffcoeff}
	\difdef{f, s, fp, c}{}{%
		outer-Ldelim=\left.,%
		outer-Rdelim=\right|,%
		sub-nudge=0mu}%
	\difdef{c}{J}{%
		op-symbol=\rmJ,%
		op-sub-nudge=-2mu}%
\usepackage[group-separator={,}]{siunitx} % To typeset numbers and units

\usepackage{xcolor}
\usepackage[pdftex]{graphicx} % Set option to dvips if the output is dvi
\usepackage{pgfplots}
	\pgfplotsset{compat=1.18}
	\usetikzlibrary{arrows}
\usepackage[export]{adjustbox}
\makeatletter%
\@ifclassloaded{IEEEtran}
	{\usepackage[caption=false,font=footnotesize]{subfig}}
	{\usepackage{subcaption}
	\captionsetup[figure]{font=footnotesize}}
\makeatother

\makeatletter%
\@ifpackageloaded{natbib}{}{%
    \usepackage[noadjust]{cite}} % Sort and properly "compress/range" citation numbers
\makeatother

\usepackage{lipsum}
\usepackage{nameref}
\usepackage{silence}
\usepackage{verbatim} % For environment {comment}
\usepackage{url} % For \url
\usepackage{enotez}
	\setenotez{
		backref=true,
		counter-format=Roman,
		split=section,
		split-title={Notes for Section <ref>},
	}

\makeatletter
\newcommand{\disablepackage}[2]{%
  \disable@package@load{#1}{#2}}
\makeatother

    \setlist{nosep,leftmargin=*}

\theoremstyle{plain}
\newtheorem{thm}{Theorem}

\newtheorem{cor}{Corollary}
\newtheorem{lem}{Lemma}

\theoremstyle{definition}
\newtheorem{dfn}{Definition}

\theoremstyle{remark}
\newtheorem{rmk}{Remark}

\newtheorem*{clm*}{Claim}

\makeatletter%
\@tempcnta=\@ne
\def\@nameedef#1{\expandafter\edef\csname #1\endcsname}
\loop\ifnum\@tempcnta<27
    \@nameedef{bb\@Alph\@tempcnta}{{\noexpand\mathbb{\@Alph\@tempcnta}}}
    \@nameedef{bf\@Alph\@tempcnta}{{\noexpand\mathbf{\@Alph\@tempcnta}}}
    \@nameedef{bm\@Alph\@tempcnta}{{\noexpand\bm{\@Alph\@tempcnta}}}
    \@nameedef{cal\@Alph\@tempcnta}{{\noexpand\mathcal{\@Alph\@tempcnta}}}
    \@nameedef{scr\@Alph\@tempcnta}{{\noexpand\mathscr{\@Alph\@tempcnta}}}
    \@nameedef{ds\@Alph\@tempcnta}{{\noexpand\mathds{\@Alph\@tempcnta}}}
    \@nameedef{frak\@Alph\@tempcnta}{{\noexpand\mathfrak{\@Alph\@tempcnta}}}
    \@nameedef{rm\@Alph\@tempcnta}{{\noexpand\mathrm{\@Alph\@tempcnta}}}
    \@nameedef{sf\@Alph\@tempcnta}{{\noexpand\mathsf{\@Alph\@tempcnta}}}
    \@nameedef{op\@Alph\@tempcnta}{{\noexpand\operatorname{\@Alph\@tempcnta}}}
    \@nameedef{hat\@Alph\@tempcnta}{{\noexpand\hat{\@Alph\@tempcnta}}}
    \@nameedef{tilde\@Alph\@tempcnta}{{\noexpand\tilde{\@Alph\@tempcnta}}}
    \@nameedef{ul\@Alph\@tempcnta}{{\noexpand\underline{\@Alph\@tempcnta}}}
    \@nameedef{ol\@Alph\@tempcnta}{{\noexpand\overline{\@Alph\@tempcnta}}}
	\@nameedef{v\@Alph\@tempcnta}{{\noexpand\vec{\@Alph\@tempcnta}}}
    \advance\@tempcnta\@ne
\repeat
\makeatother

\let\vec\undefined
\newcommand*{\vec}[1]{\bm{#1}}

\newcommand*{\Z}{\dsZ}

\newcommand*{\R}{\dsR}

\let\ForAll\forall
\let\forall\undefined
\DeclareMathOperator\forall{\ForAll}
\let\Exists\exists
\let\exists\undefined
\DeclareMathOperator\exists{\Exists}

\DeclareMathOperator{\tr}{tr}
\DeclareMathOperator{\rank}{rank}

\usepackage{empheq}
\usepackage[colorlinks,citecolor=cyan]{hyperref}

\DeclareMathOperator{\eigMax}{\overline\lambda}
\DeclareMathOperator{\eigMin}{\underline\lambda}
\DeclareMathOperator{\SVMin}{\underline\sigma}

\title{\LARGE\bf Data-driven control of linear systems using quantized data}
\author{Zhenghao Li and Guosong Yang%
\thanks{The authors are with the Department of Electrical and Computer Engineering, Rutgers University--New Brunswick, Piscataway, NJ 08854 USA (e-mails: {\tt \{zhenghao.li,\,guosong.yang\}@rutgers.edu}).}%
}
\date{}

\begin{document}
\maketitle

\begin{abstract}
This paper studies data-driven stabilization of unknown discrete-time linear systems using only quantized state measurements that take values in finite sets.
The proposed approach consists of two stages.
In the controller design stage, we collect quantized data while maintaining a prescribed quantization error bound and use them to formulate a semidefinite program (SDP).
We establish a verifiable condition under which any feasible solution to the SDP yields a stabilizing feedback gain, and show that a stabilizing gain can always be obtained when the quantization error is sufficiently small.
In the stabilization stage, a Lyapunov-based quantizer update rule is developed to guarantee exponential convergence under quantized state feedback.
As a key feature, the number of quantization cells remains finite in both stages and constant during stabilization.
Simulation results illustrate the effectiveness of the proposed approach.
\end{abstract}

%\tableofcontents

%%%%%%%%%%%%%%%%%%%%%%%%%%%%%%%%%%%%%%%%%%%%%%%%%%%%%%%%%%%%%%%%%%%%%%%%%%%%%%%%
\section{Introduction}\label{sec:intro}
Data-driven methods provide powerful tools for controlling systems where accurate models are unavailable or difficult to obtain \cite{HouWang2013,TangDaoutidis2022}.
Indirect approaches based on system identification followed by model-based control design may suffer from model mismatch, which can compromise closed-loop stability and performance guarantees.
In contrast, direct data-driven methods bypass explicit model identification and design controllers directly from data \cite{ZieglerNichols1942,Astrom1991}.
A key development in this direction is Willems' fundamental lemma \cite{WillemsRapisardaMarkovskyMoor2005}, which shows that the behavior of a controllable linear system can be fully characterized using data from a sufficiently excited trajectory.
This result has led to a behavioral framework for data-driven analysis and control of linear systems \cite{PersisTesi2020,MarkovskyDorfler2021}, with applications to linear quadratic regulation \cite{PersisTesi2020,PersisTesi2021} and model predictive control \cite{CoulsonLygerosDorfler2019}, and extensions to nonlinear systems \cite{MartinSchonAllgower2023} and switched systems \cite{RotuloPersisTesi2022}.

Modern control systems are increasingly implemented over digital communication networks, where data must be transmitted through rate-limited channels \cite{NairFagnaniZampieriEvans2007}.
This has motivated extensive research on control under data-rate constraints.
A common approach is to quantize measurements into finite sets before transmission, thereby enabling finite data rates \cite{Liberzon2003Book}. 
Foundational results using this approach are surveyed in \cite{NairFagnaniZampieriEvans2007,JiangLiu2013}, with applications in areas such as power electronics, robotics, and aerospace \cite{Fu2024}.
Recent developments span robust control \cite{ChangHuangWangLiu2020,SongChang2020,ZamaniYang2026}, event-triggered control \cite{LiuJiang2019}, and switched systems \cite{Liberzon2014,YangLiberzon2018}.

These considerations motivate the study of data-driven control under data-rate constraints.
Most existing data-driven control methods rely on unquantized measurements, which limits their applicability in communication-constrained environments.
Extensions to noisy data have introduced robustness guarantees using tools such as semidefinite programming \cite{PersisTesi2021}, the matrix S-lemma \cite{WaardeCamlibelMesbahi2022}, and data informativity \cite{WaardeEisingTrentelmanCamlibel2020}.
However, these approaches do not explicitly address data-rate constraints.

Data-driven control under quantized measurements remains relatively underexplored. 
In \cite{LiZhaoYou2024}, stabilization of unknown linear systems with quantized control inputs is studied using noisy unquantized trajectory data.
Related works consider superstabilization using partially quantized data \cite{MillerZhengSznaierHixenbaugh2024},
event-triggered control with quantized state transmission \cite{XuSunWeDuan2024}, 
and the trade-off between quantization densities for controller design and feedback stabilization \cite{TakakCetinkayaIshii2025}.
These works employ static quantizers with infinitely many levels in the feedback loop, 
a limitation explicitly identified in \cite{MillerZhengSznaierHixenbaugh2024}, 
and therefore do not achieve finite data rates.
Moreover, \cite{LiZhaoYou2024,MillerZhengSznaierHixenbaugh2024,XuSunWeDuan2024} rely at least partly on unquantized data to compute feedback gains, leaving a central challenge in data-driven control under data-rate constraints unresolved.

We develop a data-driven framework for stabilizing unknown discrete-time linear systems using only quantized state measurements that take values in finite sets.

First, in the controller design stage, quantized data are collected under a prescribed quantization error bound and used to formulate a semidefinite program (SDP). 
We establish a verifiable condition under which any feasible solution to the SDP yields a stabilizing feedback gain, together with a Lyapunov matrix used for quantizer design in the stabilization stage. 
We further show that, under controllability and persistent excitation, the SDP is feasible and every optimal solution yields a stabilizing feedback gain when the quantization error is sufficiently small.

Second, in the stabilization stage, a Lyapunov-based quantizer update rule using a fixed number of quantization cells is developed. 
We establish that, when this number is sufficiently large, the closed-loop system under quantized state feedback converges exponentially to the origin.

In both stages, the number of quantization cells is finite, and it remains constant during stabilization. 
Thus, the proposed approach provides data-driven stabilization using only quantized state measurements under finite data rates.

The remainder of the paper is organized as follows.
Section~\ref{sec:pre} introduces the system model, control objective, and quantization scheme.
Section~\ref{sec:ddc} develops data-driven controller design and establishes conditions for obtaining stabilizing feedback gains from quantized data.
Section~\ref{sec:stab} develops quantized feedback stabilization and establishes exponential convergence.
Section~\ref{sec:sim} presents simulation results, and Section~\ref{sec:end} concludes the paper.

\emph{Notation:}
For a vector $ v $, denote by % = (v_1, \ldots, v_n) \in \R^n
$ \|v\| $ % := \sqrt{v^\top v}
and
$ \|v\|_\infty $ % := \max_{1 \leq i \leq n} |v_i|
its Euclidean and infinity norms, respectively.
For a matrix $ M \in \R^{n \times m} $, denote by % = [a_{ij}]
$ \SVMin(M) $, % = \min_{\|x\| = 1} \|M x\| = \sqrt{\eigMin(M^\top M)} if n \geq m and = \min_{\|x\| = 1} \|M^\top x\| = \sqrt{\eigMin(M M^\top)} if n \leq m
$ \|M\| $, % := \max_{\|x\| = 1} \|M x\| = \sqrt{\eigMax(M^\top M)} = \sqrt{\eigMax(M M^\top)}
and
$ \|M\|_\rmF $ % := \sqrt{\sum_{i=1}^{n} \sum_{j=1}^{m} |a_{ij}|^2} = \sqrt{\tr(M^\top M)}
its smallest singular value, spectral norm (i.e., largest singular value), and Frobenius norm, respectively.
For a symmetric matrix $ S \in \R^{n \times n} $, denote by $ \eigMax(S) $ and $ \eigMin(S) $ its largest and smallest eigenvalues, respectively.
For symmetric matrices $ P, R \in \R^{n \times n} $, write $ P \succ R $ and $ P \succeq R $ when $ P - R $ is positive definite and positive semidefinite, respectively.

%%%%%%%%%%%%%%%%%%%%%%%%%%%%%%%%%%%%%%%%%%%%%%%%%%%%%%%%%%%%%%%%%%%%%%%%%%%%%%%%
\section{Problem formulation}\label{sec:pre}
Consider a discrete-time linear time-invariant (LTI) system
\begin{equation}\label{eq:sys}
	x(k+1) = A x(k) + B u(k),
	\quad
	x(0) = x_0,
	\qquad
	k \in \Z_{\geq 0},
\end{equation}
where $ x(k) \in \R^{n} $ is the state and $ u(k) \in \R^{m}$ is the control input. We assume that the dimensions $ n $ and $ m $ are known, but the matrices $ A \in \R^{n \times n} $ and $ B \in \R^{n \times m} $ are unknown. Instead, only upper bounds on their norms are known:
\begin{equation}\label{eq:sys-mx-bnd}
	\|A\| \leq \alpha,
	\qquad
	\|B\| \leq \beta.
\end{equation}
Our goal is to stabilize the unknown LTI system \eqref{eq:sys} using quantized state measurements $ \hat x(k) = q_k(x(k)) $, where $ q_k(\cdot) $ is a time-varying quantizer that maps $ \R^n $ to a \emph{finite} set of points at each time $ k $.

The proposed approach consists of two stages.
First, in the \emph{controller design stage}, we collect input and quantized state data and use them to compute a stabilizing feedback gain.
The quantizer update rule in this stage is designed to maintain a prescribed quantization error bound.
Second, in the \emph{stabilization stage}, we employ a Lyapunov-based quantizer update rule with a fixed number of quantization cells and establish exponential convergence under quantized state feedback.
The controller design stage need not be performed offline.
For concreteness, we consider stabilization beginning directly after controller design without resetting the system; 
the designed controller can also be applied after a system reset or at a later time.

The quantization scheme is introduced in Section~\ref{ssec:pre-quant}. The design and stabilization stages are developed in Sections~\ref{sec:ddc} and~\ref{sec:stab}, respectively.

%%%%%%%%%%%%%%%%%%%%%%%%%%%%%%%%%%%%%%%%
\subsection{Quantization scheme}\label{ssec:pre-quant}
The initial state $ x_0 $ is unknown. We assume that there exists a known bound $ E_0 > 0 $ such that
\[
	\|x_0\|_\infty \leq E_0.
\]
If unavailable, such a bound can be obtained in finite time using search techniques in quantized control; see, e.g., \cite{Liberzon2003Book}.

At each time $ k $, we maintain a bound $ E_k $ such that
\begin{equation}\label{eq:quant-bnd}
	\|x(k)\|_\infty \leq E_k,
\end{equation}
that is, $ x(k) $ lies in the hypercube
\[
	\calR(E_k) := \{v \in \R^n: \|v\|_\infty \leq E_k\},
\]
which serves as the quantization range.
Given $ E_k $, quantizer update rules are designed to determine $ E_{k+1} $ such that $ x(k+1) \in \calR(E_{k+1}) $. 
The update rules differ between the two stages and are derived in Sections~\ref{ssec:ddc-quant} and~\ref{ssec:stab-quant}.

The quantizer partitions $ \calR(E_k) $ into $ N_k^n $ equal hypercubic cells, with $ N_k $ cells per dimension, where $ N_k $ is a positive integer. The quantized measurement $ \hat x(k) $ is defined as the center of the cell containing $ x(k) $. Consequently, the quantization error
\[ 
	d(k) := \hat x(k) - x(k)
\]
satisfies
\begin{equation}\label{eq:quant-err-bnd}
	\|d(k)\|_\infty
	= \|\hat x(k) - x(k)\|_\infty
	\leq \frac{E_k}{N_k}.
\end{equation}
Since the quantizer takes values in a set of cardinality $ N_k^n $, the measurement at time $ k $ can be represented using a finite number of bits.

%%%%%%%%%%%%%%%%%%%%%%%%%%%%%%%%%%%%%%%%%%%%%%%%%%%%%%%%%%%%%%%%%%%%%%%%%%%%%%%%
\section{Data-driven controller design}\label{sec:ddc}

%%%%%%%%%%%%%%%%%%%%%%%%%%%%%%%%%%%%%%%%
\subsection{Data-driven control preliminaries}\label{ssec:ddc-pre}
During the controller design stage, we excite the unknown LTI system \eqref{eq:sys} with an input sequence $ \{u(k)\}_{k=0}^{T-1} $ satisfying
\begin{equation}\label{eq:input-bnd}
	\|u(k)\| \leq \bar u,
	\qquad
	\forall k \in \{0, 1, \ldots, T-1\},
\end{equation}
and collect the quantized state measurement sequence $ \{\hat x(k)\}_{k=0}^{T} $.
Define the corresponding data matrices
\begin{equation}\label{eq:data-mx-dfn}
\begin{aligned}
	\hat X_0 &:=
	\begin{bmatrix}
		\hat x(0) & \hat x(1) & \cdots & \hat x(T-1)
	\end{bmatrix}, \\
	\hat X_1 &:=
	\begin{bmatrix}
		\hat x(1) & \hat x(2) & \cdots & \hat x(T)
	\end{bmatrix}.
\end{aligned}
\end{equation}
The data matrices $ U_0 \in \R^{m \times T} $, $ X_0, X_1 \in \R^{n \times T} $, and $ D_0, D_1 \in \R^{n \times T} $ for the input $ u(k) $, unquantized state $ x(k) $, and quantization error $ d(k) $, respectively, are defined analogously. The matrices $ X_0, X_1, D_0, D_1 $ are introduced only for analysis and are not available for control synthesis.

A key condition for computing a stabilizing controller is
\begin{equation}\label{eq:rank-cond-quant}
	\rank \begin{bmatrix} 
		U_0 \\ \hat X_0 
	\end{bmatrix} 
	= n + m.
\end{equation}
In the absence of quantization, it is known from \cite{WillemsRapisardaMarkovskyMoor2005} that this rank condition holds when the input sequence is sufficiently exciting.
\begin{dfn}[\cite{WillemsRapisardaMarkovskyMoor2005}]\label{dfn:pe}
The input sequence $ \{u(k)\}_{k=0}^{T-1} $ is \emph{persistently exciting (PE)} of order $ L $ if the Hankel matrix
\begin{equation*}\label{eq:Hankel}
	\scrU_{0, L, T-L+1} :=
	\begin{bmatrix}
		u(0) & u(1) & \cdots & u(T-L) \\
		u(1) & u(2) & \cdots & u(T-L+1) \\
		\vdots & \vdots & \ddots & \vdots \\
		u(L-1) & u(L) & \cdots & u(T-1)
	\end{bmatrix}
\end{equation*}
has full row rank $ m L $.
\end{dfn}
\begin{lem}[{\cite[Cor.~2]{WillemsRapisardaMarkovskyMoor2005}}]\label{lem:fund-lem}
If the LTI system \eqref{eq:sys} is controllable and the input sequence $ \{u(k)\}_{k=0}^{T-1} $ is PE of order $ n + 1 $, then
\begin{equation}\label{eq:rank-cond}
	\rank \begin{bmatrix} 
		U_0 \\ X_0 
	\end{bmatrix}
	= n + m.
\end{equation}
\end{lem}
We assume
\begin{equation}\label{eq:ddc-T-cond}
	T \geq m (n+1) + n,
\end{equation}
which is necessary for $ \{u(k)\}_{k=0}^{T-1} $ to be PE of order $ n + 1 $.
Theorem~\ref{thm:ddc-feas} below establishes that, under the assumptions of Lemma~\ref{lem:fund-lem}, the quantized rank condition \eqref{eq:rank-cond-quant} also holds when the quantization errors are sufficiently small.

%%%%%%%%%%%%%%%%%%%%%%%%%%%%%%%%%%%%%%%%
\subsection{Quantized data collection}\label{ssec:ddc-quant}
For $ k < T $, by \eqref{eq:sys-mx-bnd}, \eqref{eq:quant-bnd}, and \eqref{eq:input-bnd},
\begin{align*}
	\|x(k+1)\|_\infty 
	&= \|A x(k) + B u(k)\|_\infty \\
	&\leq \|A x(k) + B u(k)\| \\
	&\leq \alpha \|x(k)\| + \beta \bar u \\
	&\leq \alpha \sqrt{n} \|x(k)\|_\infty + \beta \bar u \\
	&\leq \alpha \sqrt{n} E_k + \beta \bar u.
\end{align*}
To ensure that \eqref{eq:quant-bnd} holds for all $ k \leq T $, we let
\begin{equation}\label{eq:quant-bnd-ddc}
	E_{k+1} := \alpha \sqrt{n} E_k + \beta \bar u,
	\qquad
	k = 0, 1, \ldots, T-1.
\end{equation}
To maintain a prescribed quantization error bound $ \delta > 0 $, we let
\begin{equation}\label{eq:quant-lvl-ddc}
	N_k := \left\lceil \frac{E_k}{\delta} \right\rceil,
	\qquad
	k = 0, 1, \ldots, T,
\end{equation}
where $ \lceil \cdot \rceil $ is the ceiling function.
Then \eqref{eq:quant-err-bnd} gives
\begin{equation}\label{eq:quant-err-bnd-ddc}
	\|d(k)\|_\infty \leq \delta
	\qquad
	\forall k \in \{0, 1, \ldots, T\}.
\end{equation}

The update rule \eqref{eq:quant-bnd-ddc} yields a bound on $ \|x(k+1)\|_\infty $ that is independent of the collected data but can lead to exponential growth of $ E_k $ and $ N_k $ during this stage.
A tighter update rule can be obtained by exploiting the available quantized measurement $ \hat x(k) $ and letting
\begin{equation}\label{eq:quant-bnd-ddc-data}
	E_{k+1} := \alpha (\|\hat x(k)\| + \sqrt{n} \delta) + \beta \bar u, 
	\qquad
	k = 0, 1, \ldots, T-1.
\end{equation}
Indeed, \eqref{eq:quant-err-bnd-ddc} implies
\begin{align*}
	\|x(k)\|
	&\leq \|\hat x(k)\| + \|d(k)\| \\
	&\leq \|\hat x(k)\| + \sqrt{n} \|d(k)\|_\infty \\
	&\leq \|\hat x(k)\| + \sqrt{n} \delta.
\end{align*}
Therefore, \eqref{eq:quant-bnd-ddc-data} also ensures that \eqref{eq:quant-bnd} holds for all $ k \leq T $.
In practice, this update can significantly reduce the number of quantization cells.

Although the number of quantization cells may vary during this stage, $ E_k $ and $ N_k $ can be determined at each time from their update rules and, when \eqref{eq:quant-bnd-ddc-data} is used, previously transmitted quantized measurements; 
hence they need not be communicated separately.
Thus, only the quantizer output at each time needs to be transmitted, and it can be represented using a finite number of bits.
The number of quantization cells remains constant during the stabilization stage in Section~\ref{sec:stab}.

%%%%%%%%%%%%%%%%%%%%%%%%%%%%%%%%%%%%%%%%
\subsection{Data-driven controller synthesis}\label{ssec:ddc-opt}
We adopt the optimization-based framework in \cite{PersisTesi2021} for stabilizing unknown LTI systems in the presence of disturbances and extend it to quantized state measurements.

In the absence of quantization, a stabilizing controller can be computed by solving the following \emph{semidefinite program (SDP)}:
\begin{subequations}\label{eq:sdp-unquant}
\begin{align}
	&\min_{(\gamma, Q, P, L)} \gamma \label{eq:sdp-unquant-val}\\
	&\text{subject to} \notag\\
	&X_1 Q P^{-1} Q^{\top} X_1^{\top} - P + I \preceq 0, \label{eq:sdp-unquant-lya}\\
	&P \succeq I, \label{eq:sdp-unquant-p}\\
	&L - U_0 Q P^{-1} Q^{\top} U_0^{\top} \succeq 0, \label{eq:sdp-unquant-l}\\
	&X_0 Q = P, \label{eq:sdp-unquant-q}\\
	&\tr(P) + \tr(L) \leq \gamma, \label{eq:sdp-unquant-gamma}
\end{align}
\end{subequations}
where $ U_0 $ and $ X_0, X_1 $ are the input and unquantized state data matrices introduced in Section~\ref{ssec:ddc-pre}.
The relevant properties are summarized in the next lemma, which follows from \cite[Lem.~2 and~3]{PersisTesi2021}. 
A proof is provided here for completeness.
\begin{lem}
\begin{enumerate}
	\item Let $ (\gamma, Q, P, L) $ be any feasible solution to \eqref{eq:sdp-unquant}, and define $ K := U_0 Q P^{-1} $. Then $ A + B K $ is Schur stable (i.e., all eigenvalues lie in the open unit disk).
	\item Suppose that the rank condition \eqref{eq:rank-cond} holds. For any $ K $ such that $ A + B K $ is Schur stable, \eqref{eq:sdp-unquant} admits a feasible solution $ (\gamma, Q, P, L) $ satisfying $ K = U_0 Q P^{-1} $.
\end{enumerate}
\end{lem}
\begin{proof}
\begin{enumerate}[wide]
\item
By \eqref{eq:sdp-unquant-q} and $ U_0 Q = K P $,
\[
	X_1 Q = (A X_0 + B U_0) Q = (A + B K) P.
\]
Hence \eqref{eq:sdp-unquant-lya} implies
\[
	(A + B K) P (A + B K)^\top - P + I \preceq 0,
\]
where $ P \succ 0 $. Therefore, $ A + B K $ is Schur stable.

\item
Since $ A + B K $ is Schur stable, there exists a unique $ P \succ 0 $ satisfying the discrete-time Lyapunov equation
\[ (A + B K) P (A + B K)^\top - P + I = 0. \]
Then \eqref{eq:sdp-unquant-p} holds as $ P = I + (A + B K) P (A + B K)^\top \succeq I $.
Let
\[
	Z := \begin{bmatrix} U_0 \\ X_0 \end{bmatrix}.
\]
Under \eqref{eq:rank-cond}, $ Z $ has full row rank, so its pseudoinverse $ Z^+ $ satisfies $ Z^+ = Z^\top (Z Z^\top)^{-1} $ and hence $ Z Z^+ = I $.
Define
\[
	Q := Z^+ \begin{bmatrix} K \\ I \end{bmatrix} P.
\]
Then
\[
	\begin{bmatrix} U_0 \\ X_0 \end{bmatrix} Q
	= Z Z^+ \begin{bmatrix} K \\ I \end{bmatrix} P
	= \begin{bmatrix} K \\ I \end{bmatrix} P,
\]
which implies \eqref{eq:sdp-unquant-q} and
\begin{align*}
	X_1 Q
	& = (A X_0 + B U_0) Q \\
	&= (A + B K) P.
\end{align*}
Thus, \eqref{eq:sdp-unquant-lya} holds with equality.
Finally, choosing
\[
	L := K P K^\top,
	\qquad
	\gamma := \tr(P) + \tr(L)
\]
satisfies \eqref{eq:sdp-unquant-l} and \eqref{eq:sdp-unquant-gamma} with equality. 
Therefore, $ (\gamma, Q, P, L) $ is feasible and satisfies $ K = U_0 Q P^{-1} $.
\qedhere
\end{enumerate}
\end{proof}

When only quantized state measurements are available, we replace $ X_0, X_1 $ in \eqref{eq:sdp-unquant} with the quantized state data matrices $ \hat X_0, \hat X_1 $ defined in \eqref{eq:data-mx-dfn} and consider the following \emph{quantized SDP}:
\begin{subequations}\label{eq:sdp-quant}
\begin{align}
	&\min_{(\gamma, Q, P, L, W)} \gamma \label{eq:sdp-quant-val}\\
	&\text{subject to} \notag\\
	&\hat X_1 Q P^{-1} Q^{\top} \hat X_1^{\top} - P + I \preceq 0, \label{eq:sdp-quant-lya}\\
	&P \succeq I, \label{eq:sdp-quant-p}\\
	&L - U_0 Q P^{-1} Q^{\top} U_0^{\top} \succeq 0, \label{eq:sdp-quant-l}\\
	&\hat X_0 Q = P, \label{eq:sdp-quant-q}\\
	&W - Q P^{-1} Q^\top \succeq 0, \label{eq:sdp-quant-w}\\
	&\tr(P) + \tr(L) + \omega \tr(W) \leq \gamma. \label{eq:sdp-quant-gamma}
\end{align}
\end{subequations}
The additional variable $ W $ and parameter $ \omega > 0 $ introduce a regularization term motivated by robustness to disturbances; see \cite[Sec.~5]{PersisTesi2021}.
Larger values of $ \omega $ encourage solutions with smaller $ \tr(Q P^{-1} Q^\top) $, which can make the sufficient condition in Theorem~\ref{thm:ddc-stab} below easier to satisfy.

The data matrices satisfy
\[
	X_1 = A X_0 + B U_0, 
	\quad 
	\hat X_0 = X_0 + D_0, 
	\quad 
	\hat X_1 = X_1 + D_1.
\]
Hence
\[
	\hat{X}_1 
	= A X_0 + B U_0 + D_1 
	= A \hat{X}_0 + B U_0 + (D_1 - A D_0).
\]
Let
\[
	\Delta_0 := D_1 - A D_0 \in \R^{n \times T},
\]
which captures the discrepancy introduced by the quantized state measurements.
By \eqref{eq:sys-mx-bnd} and \eqref{eq:quant-err-bnd-ddc}, the columns of $ \Delta_0 $ satisfy
\begin{align*}
	\|d(k+1) - A d(k)\| 
	&\leq \|d(k+1)\| + \alpha \|d(k)\| \\
	&\leq \sqrt{n} \|d(k+1)\|_\infty + \alpha \sqrt{n} \|d(k)\|_\infty \\
	&\leq \sqrt{n} (\alpha + 1) \delta
\end{align*}
for all $ k \in \{0, 1, \ldots, T-1\} $. Hence
\begin{equation}\label{eq:quant-err-mx-bnd}
	\|\Delta_0\| 
	\leq \|\Delta_0\|_\rmF 
	\leq \sqrt{n T} (\alpha + 1) \delta 
	=: \hat\delta,
\end{equation}
where $ \hat\delta $ can be made arbitrarily small by decreasing the prescribed quantization error bound $ \delta $, at the expense of requiring more quantization cells.

Our first theorem establishes a verifiable condition under which any feasible solution to the quantized SDP \eqref{eq:sdp-quant} yields a stabilizing controller.
\begin{thm}\label{thm:ddc-stab}
Let $ (\bar\gamma, \bar Q, \bar P, \bar L, \bar W) $ be any feasible solution to \eqref{eq:sdp-quant}, and define
\begin{equation}\label{eq:ddc-stab-gain}
	\bar K := U_0 \bar Q \bar P^{-1},
	\qquad
	\bar M := \bar Q \bar P^{-1} \bar Q^\top.
\end{equation}
Then
\begin{equation}\label{eq:ddc-stab}
	(A + B \bar K) \bar P (A + B \bar K)^\top - \bar P \preceq -\eta_1 I,
\end{equation}
where
\begin{equation}\label{eq:ddc-stab-coef}
	\eta_1 := 1 - \hat\delta^2 \|\bar M\| - 2 \hat\delta \|\hat X_1 \bar M\|.
\end{equation}
If
\begin{equation}\label{eq:ddc-stab-cond}
	\eta_1 > 0,
\end{equation}
then $ A + B \bar K $ is Schur stable.
\end{thm}
\begin{proof}
Since $ \bar P = \hat X_0 \bar Q $ and $ \bar K \bar P = U_0 \bar Q $,
\begin{align*}
	(A + B \bar K) \bar P 
	&= (A \hat X_0 + B U_0) \bar Q \\
	&= (\hat X_1 - \Delta_0) \bar Q.
\end{align*}
Then
\begin{align*}
	&\quad\,\, (A + B \bar K) \bar P (A + B \bar K)^\top \\
	&= (\hat X_1 - \Delta_0) \bar Q \bar P^{-1} \bar Q^\top (\hat X_1 - \Delta_0)^\top \\
	&= \hat X_1 \bar M \hat X_1^{\top} + \bar\Psi,
\end{align*}
where
\begin{equation*}%\label{eq:ddc-stab-res}
	\bar\Psi := 
	\Delta_0 \bar M \Delta_0^\top 
	- \hat X_1 \bar M \Delta_0^\top 
	- \Delta_0 \bar M \hat X_1^\top.
\end{equation*}
Combining this with \eqref{eq:sdp-quant-lya} yields
\begin{equation*}
	(A + B \bar K) \bar P (A + B \bar K)^\top - \bar P 
	\preceq \bar\Psi - I.
\end{equation*}
Using \eqref{eq:quant-err-mx-bnd}, we obtain
\begin{align*}
	\|\bar\Psi\| 
	&\leq \|\Delta_0\|^2 \|\bar M\| + 2 \|\Delta_0\| \|\hat X_1 \bar M\| \\
	&\leq \hat\delta^2 \|\bar M\| + 2 \hat\delta \|\hat X_1 \bar M\| \\
	&= 1 - \eta_1.
\end{align*}
Since $ \bar\Psi $ is symmetric,
\[
	\bar\Psi \preceq \|\bar\Psi\| I \preceq (1 - \eta_1) I,
\]
and hence \eqref{eq:ddc-stab} holds.
If $ \eta_1 > 0 $, then \eqref{eq:ddc-stab} implies
\begin{equation*}%\label{eq:ddc-stab-lya}
	(A + B \bar K) \bar P (A + B \bar K)^\top - \bar P \prec 0,
\end{equation*}
where $ \bar P \succ 0 $. Therefore, $ A + B \bar K $ is Schur stable.
\end{proof}

In practice, the feedback gain $ \bar K $ can be computed by solving \eqref{eq:sdp-quant} using the available input and quantized state data.
The sufficient condition \eqref{eq:ddc-stab-cond} can also be verified directly from these data.

Our second theorem guarantees that a stabilizing controller can be obtained from any optimal solution when the quantization errors are sufficiently small.

\begin{thm}\label{thm:ddc-feas}
Suppose that the LTI system \eqref{eq:sys} is controllable and the input sequence $ \{u(k)\}_{k=0}^{T-1} $ is PE of order $ n + 1 $.
There exists $ \delta^* > 0 $ such that, if the quantization error bound $ \delta < \delta^* $, then the rank condition \eqref{eq:rank-cond-quant} holds, the quantized SDP \eqref{eq:sdp-quant} is feasible, and every optimal solution to \eqref{eq:sdp-quant} yields a stabilizing controller.
\end{thm}

Theorem~\ref{thm:ddc-feas} is restricted to optimal solutions to the quantized SDP \eqref{eq:sdp-quant}, in contrast to Theorem~\ref{thm:ddc-stab}, which applies to all feasible solutions.
This restriction arises because the proof relies on comparing the optimal value of \eqref{eq:sdp-quant} with the objective value of a specifically constructed feasible solution.
The full proof is presented in Sections~\ref{sssec:ddc-feas-rank}--\ref{sssec:ddc-feas-stab}, including an explicit sufficient bound $ \delta^* $ in \eqref{eq:ddc-feas-cond}.

We also establish a stronger result under an alternative, more conservative bound.

\begin{cor}\label{cor:ddc-feas}
Under the assumptions of Theorem~\ref{thm:ddc-feas}, there exists $ \bar\delta^* > 0 $ such that, if $ \delta < \bar\delta^* $, then the rank condition \eqref{eq:rank-cond-quant} holds, the quantized SDP \eqref{eq:sdp-quant} is feasible, and every optimal solution to \eqref{eq:sdp-quant} satisfies the verifiable condition \eqref{eq:ddc-stab-cond}.
\end{cor}

The proof of Corollary~\ref{cor:ddc-feas} builds on the proof of Theorem~\ref{thm:ddc-feas} and is completed in Section~\ref{sssec:ddc-feas-stab-cond}, including an explicit formula for $ \bar\delta^* $ in \eqref{eq:ddc-feas-cond-alt}.

Both bounds $ \delta^* $ and $ \bar\delta^* $ depend on the unknown matrices $ A $ and $ B $ and therefore cannot be verified directly from the available data.
Nevertheless, they provide an explicit guarantee that, under sufficiently fine quantization, a stabilizing controller can always be obtained, and the verifiable condition \eqref{eq:ddc-stab-cond} can always be satisfied.

%%%%%%%%%%%%%%%%%%%%%%%%%%%%%%%%%%%%%%%%
\subsection{Proof of Theorem~\ref{thm:ddc-feas}}\label{ssec:ddc-feas-pf}
The proof consists of three steps. We first establish the rank condition \eqref{eq:rank-cond-quant}, then feasibility of the quantized SDP \eqref{eq:sdp-quant}, and finally that every optimal solution to \eqref{eq:sdp-quant} yields a stabilizing controller.

%%%%%%%%%%%%%%%%%%%%
\subsubsection{Error bound for rank condition}\label{sssec:ddc-feas-rank}
Let
\[
	\hat Z = 
	\begin{bmatrix} 
		U_0 \\ \hat X_0 
	\end{bmatrix}, 
	\qquad 
	Z = 
	\begin{bmatrix} 
		U_0 \\ X_0 
	\end{bmatrix}.
\]
In the first step, we derive a sufficient condition for \eqref{eq:rank-cond-quant} to hold, or equivalently, for $ \SVMin(\hat Z) > 0 $. 

Since
\[
	\hat Z = Z + \begin{bmatrix} 0 \\ D_0 \end{bmatrix},
\]
Weyl's inequality for singular values \cite[Th.~3.3.16, p.~178]{HornJohnson1991} implies
\[
	\SVMin(\hat Z) \geq \SVMin(Z) - \|D_0\|.
\]
%----
%
We first bound $ \SVMin(Z) $.
Since the input sequence $ \{u(k)\}_{k=0}^{T-1} $ is PE of order $ n + 1 $, the Hankel matrix $ \scrU_{0, n+1, T-n} $ defined in Definition~\ref{dfn:pe} has full row rank $ m (n + 1) $, so $ \SVMin(\scrU_{0, n+1, T-n}) > 0 $.
Moreover, since the pair $ (A, B) $ is controllable, it follows from the proof of \cite[Th.~3]{CoulsonWaardeLygerosDorfler2023} that
\[
	\SVMin(Z) \geq 
	\frac{\rho_0}{\sqrt{n+1}} \SVMin(\scrU_{0, n+1, T-n}),
\]
where $ \rho_0 > 0 $ is a constant depending on $ A $ and $ B $.
We next bound $ \|D_0\| $. 
Using \eqref{eq:quant-err-bnd-ddc}, we obtain
\begin{align*}
	\|D_0\| 
	\leq \|D_0\|_\rmF 
	\leq \sqrt{\sum_{k=0}^{T-1} n \|d(k)\|_\infty^2} 
	\leq \sqrt{nT} \delta.
\end{align*}

Combining these bounds yields
\[
	\SVMin(\hat Z) 
	\geq \frac{\rho_0}{\sqrt{n+1}} \SVMin(\scrU_{0, n+1, T-n}) - \sqrt{nT} \delta.
\]
Therefore, if
\begin{equation}\label{eq:ddc-feas-cond-0}
	\delta 
	< \frac{\rho_0}{\sqrt{(n+1) n T}} \SVMin(\scrU_{0, n+1, T-n})
	=: \delta_0,
\end{equation}
then
\begin{equation}\label{eq:ddc-feas-cond-0-sv}
	\SVMin(\hat Z) \geq \sqrt{n T} (\delta_0 - \delta) > 0,
\end{equation}
and hence the rank condition \eqref{eq:rank-cond-quant} holds.

%%%%%%%%%%%%%%%%%%%%
\subsubsection{Error bound for feasibility of quantized SDP}\label{sssec:ddc-feas-feas}
In the second step, we derive a sufficient condition under which a feasible solution to \eqref{eq:sdp-quant} can be constructed.

Suppose for now that \eqref{eq:ddc-feas-cond-0} holds.
Since $ (A, B) $ is controllable, there exists a feedback gain $ K_* $ such that $ A + B K_* $ is Schur stable.
Let $ P_* \succ 0 $ be the unique solution to the discrete-time Lyapunov equation
\[
	(A + B K_*) P_* (A + B K_*)^\top - P_* + I = 0,
\]
and define
\begin{align*}
	Q_* 
	&:= \hat Z^+ \begin{bmatrix} K_* \\ I \end{bmatrix} P_*,
	&L_* 
	&:= K_* P_* K_*^\top, \\
	\gamma_* 
	&:= \tr(P_*) + \tr(L_*),
	&W_* 
	&:= Q_* P_*^{-1} Q_*^\top.
\end{align*}
Under \eqref{eq:rank-cond-quant}, $ \hat Z $ has full row rank, so its pseudoinverse $ \hat Z^+ $ satisfies $ \hat Z \hat Z^+ = I $. 
Let $ \eta_2 \in (0, 1) $ and consider the tuple
\begin{equation}\label{eq:ddc-feas-feas-tuple}
	\left( \frac{\gamma_* + \omega \tr(W_*)}{\eta_2}, 
	\frac{Q_*}{\eta_2}, 
	\frac{P_*}{\eta_2}, 
	\frac{L_*}{\eta_2}, 
	\frac{W_*}{\eta_2} 
	\right).
\end{equation} 
By construction, this tuple satisfies all constraints in \eqref{eq:sdp-quant} except possibly \eqref{eq:sdp-quant-lya}.

Using
\begin{align*}
	(\hat X_1 - \Delta_0) Q_*
	&= (A \hat X_0 + B U_0) Q_* \\ 
	&= \begin{bmatrix} B & A \end{bmatrix} \hat Z \hat Z^+
	\begin{bmatrix} K_* \\ I \end{bmatrix} P_* \\
	&= (A + B K_*) P_*,
\end{align*}
we obtain
\begin{align*}
	&\quad\,\ \hat X_1 Q_* P_*^{-1} Q_*^{\top} \hat X_1^\top \\
	&= (\hat X_1 - \Delta_0) Q_* P_*^{-1} Q_*^{\top} (\hat X_1 - \Delta_0)^{\top} - \Psi_* \\
	&= (A + B K_*) P_* (A + B K_*)^\top - \Psi_* \\
	&= P_* - I - \Psi_*,
\end{align*}
where
\[
	\Psi_* := 
	\Delta_0 W_* \Delta_0^\top 
	- \hat X_1 W_* \Delta_0^\top 
	- \Delta_0 W_* \hat X_1^\top.
\]
Hence
\[
	\hat X_1 
	\left( \frac{Q_*}{\eta_2} \right) 
	\left( \frac{P_*}{\eta_2} \right)^{-1} 
	\left( \frac{Q_*}{\eta_2} \right)^{\top} 
	\hat X_1^\top 
	- \left( \frac{P_*}{\eta_2} \right) + I 
	= \frac{(\eta_2 - 1) I - \Psi_*}{\eta_2}.
\]
Therefore, \eqref{eq:sdp-quant-lya} holds if
\[
	 -\Psi_* \preceq (1 - \eta_2) I.
\]

To derive a sufficient condition in terms of $ \delta $, rewrite
\begin{align*}
	-\Psi_*
	&= 
	-\Delta_0 W_* \Delta_0^\top 
	+ \hat X_1 W_* \Delta_0^\top 
	+ \Delta_0 W_* \hat X_1^\top \\
	&= 
	\Delta_0 W_* \Delta_0^\top 
	+ (\hat X_1 - \Delta_0) W_* \Delta_0^\top 
	+ \Delta_0 W_* (\hat X_1 - \Delta_0)^\top.
\end{align*}
Define 
\begin{equation*}
	\Xi := \begin{bmatrix} K_* \\ I \end{bmatrix} P_* \begin{bmatrix} K_* \\ I \end{bmatrix}^\top.
\end{equation*}
Then
\[
	W_* = \hat Z^+ \Xi (\hat Z^+)^\top,
\]
and
\begin{align*}
	(\hat X_1 - \Delta_0) W_* 
	&= (A \hat X_0 + B U_0) W_* \\
	&= \begin{bmatrix} B & A \end{bmatrix} \hat Z \hat Z^+ \Xi (\hat Z^+)^\top \\
	&= \begin{bmatrix} B & A \end{bmatrix} \Xi (\hat Z^+)^\top.
\end{align*}
Moreover, the derivation of \eqref{eq:quant-err-mx-bnd} also yields
\begin{equation}\label{eq:quant-err-mx-bnd-A}
	\|\Delta_0\| \leq \sqrt{n T} (\|A\| + 1) \delta,
\end{equation}
and \eqref{eq:ddc-feas-cond-0-sv} implies
\begin{equation}\label{eq:data-mx-bnd}
	\|\hat Z^+\| 
	= \frac{1}{\SVMin(\hat Z)} 
	\leq \frac{1}{\sqrt{n T} (\delta_0 - \delta)}.
\end{equation}
Combining these bounds and letting
\[
	\overline\sigma_{AB} := \left\| \begin{bmatrix} B & A \end{bmatrix} \right\|,
\]
we obtain
\begin{align*}
	\|{-\Psi_*}\|
	&\leq \|\Xi\| \|\hat Z^+\|^2  \|\Delta_0\|^2
	+ 2 \left\| \begin{bmatrix} B & A \end{bmatrix} \right\| 
	\|\Xi\| \|\hat Z^+\| \|\Delta_0\| \\
	&\leq \|\Xi\| \left( \frac{(\|A\| + 1) \delta}{\delta_0 - \delta} \right)^2 
	+ 2 \overline\sigma_{AB} \|\Xi\| \frac{(\|A\| + 1) \delta}{\delta_0 - \delta} \\
	&= \|\Xi\| \left( \frac{(\|A\| + 1) \delta}{\delta_0 - \delta} + \overline\sigma_{AB} \right)^2 
	- \overline\sigma_{AB}^2 \|\Xi\|.
\end{align*}

Requiring the above expression to be no greater than $ 1 - \eta_2 $ and solving for $ \delta $ yields the following sufficient condition.
If
\begin{equation}\label{eq:ddc-feas-cond-1}
	\delta \leq 
	\frac{\delta_0}
	{1 + \frac{\|A\| + 1}
		{\sqrt{\overline\sigma_{AB}^2 + \frac{1 - \eta_2}{\|\Xi\|}} 
		- \overline\sigma_{AB}}
	}
	=: \delta_1,
\end{equation}
then \eqref{eq:ddc-feas-cond-0} also holds, and $ \|{-\Psi_*}\| \leq 1 - \eta_2 $.
Since $ -\Psi_* $ is symmetric,
\[
	-\Psi_*
	\preceq \|{-\Psi_*}\| I
	\preceq (1 - \eta_2) I.
\]
Consequently, the tuple \eqref{eq:ddc-feas-feas-tuple} satisfies all constraints in \eqref{eq:sdp-quant}.
Therefore, the quantized SDP \eqref{eq:sdp-quant} is feasible.

%%%%%%%%%%%%%%%%%%%%
\subsubsection{Error bound for stabilizing controllers from optimal solutions}\label{sssec:ddc-feas-stab}
Let $ (\bar\gamma, \bar Q, \bar P, \bar L, \bar W) $ be an arbitrary optimal solution to \eqref{eq:sdp-quant}, and define 
$ \bar K $
and
$ \bar M $
as in \eqref{eq:ddc-stab-gain}.
In the third step, we derive a sufficient condition, independent of the particular optimal solution, under which $ A + B \bar K $ is Schur stable.

Suppose for now that \eqref{eq:ddc-feas-cond-1} holds and the feasible solution \eqref{eq:ddc-feas-feas-tuple} to \eqref{eq:sdp-quant} is constructed as in the previous step.
Arguments from the proof of Theorem~\ref{thm:ddc-stab} yield
\begin{align*}
	A + B \bar K 
	= (\hat X_1 - \Delta_0) \bar Q \bar P^{-1},
\end{align*}
and
\[
	(A + B \bar K) \bar P (A + B \bar K)^\top - \bar P 
	\preceq \bar\Psi - I,
\]
where
\[
	\bar\Psi := 
	\Delta_0 \bar M \Delta_0^\top 
	- \hat X_1 \bar M \Delta_0^\top 
	- \Delta_0 \bar M \hat X_1^\top.
\]
Thus, $ A + B \bar K $ is Schur stable if $ \bar\Psi \prec I $.

To derive a sufficient condition in terms of $ \delta $, rewrite
\begin{align*}
	\bar\Psi 
	&= 
	-\Delta_0 \bar M \Delta_0^\top 
	- (\hat X_1 - \Delta_0) \bar M \Delta_0^\top 
	- \Delta_0 \bar M (\hat X_1 - \Delta_0)^\top \\
	&= 
	-\Delta_0 \bar M \Delta_0^\top 
	- (A + B \bar K) \bar Q^\top \Delta_0^\top 
	- \Delta_0 \bar Q (A + B \bar K)^\top \\
	&\preceq 
	-(A + B \bar K) \bar Q^\top \Delta_0^\top 
	- \Delta_0 \bar Q (A + B \bar K)^\top.
\end{align*}
We next use optimality to obtain bounds on $ \bar K $ and $ \bar Q $.
Since \eqref{eq:ddc-feas-feas-tuple} is feasible,
\[
	\bar\gamma 
	\leq 
	\frac{\gamma_* + \omega \tr(W_*)}{\eta_2}.
\]
Using $ W_* = \hat Z^+ \Xi (\hat Z^+)^\top $, $ \Xi \succeq 0 $, and \eqref{eq:data-mx-bnd}, we obtain
\begin{align*}
	\tr(W_*)
	&= 
	\tr \left( \Xi (\hat Z^+)^\top \hat Z^+ \right) \\
	&\leq 
	\tr(\Xi) \left\| (\hat Z^+)^\top \hat Z^+ \right\| \\ % Ruhe's trace inequality \cite[Th.~4.3.53, p.~255]{HornJohnson2013}
	&= 
	\tr(\Xi) \|\hat Z^+\|^2 \\
	&\leq
	\frac{\tr(\Xi)}{n T (\delta_0 - \delta)^2}.
\end{align*}
Define
\begin{equation}\label{eq:chi}
	\chi := \gamma_* + \frac{\omega \tr(\Xi)}{n T (\delta_0 - \delta_1)^2}.
\end{equation}
Since $ \delta \leq \delta_1 $,
\[
	\bar\gamma 
	\leq 
	\frac{\gamma_* + \omega \tr(W_*)}{\eta_2} 
	\leq 
	\frac{\chi}{\eta_2}.
\]
By \eqref{eq:sdp-quant-l}, \eqref{eq:sdp-quant-w}, and
\eqref{eq:sdp-quant-gamma},
\[
	\tr(\bar P) 
	+ \tr(\bar K \bar P \bar K^\top) 
	+ \omega \tr(\bar Q \bar P^{-1} \bar Q^\top) 
	\leq 
	\bar\gamma.
\]
Since $ \bar P \succeq I $, we have
\begin{align*}
	\tr(\bar P) + \tr(\bar K \bar P \bar K^\top)
	&= \tr(\bar P) + \tr(\bar P \bar K^\top \bar K) \\ % Ruhe's trace inequality \cite[Th.~4.3.53, p.~255]{HornJohnson2013}
	&\geq n + \SVMin(\bar P) \tr(\bar K^\top \bar K) \\
	&= n + \SVMin(\bar P) \|\bar K\|_\rmF^2 \\
	&\geq 1 + \|\bar K\|^2,
\end{align*}
and 
\begin{align*}
	\tr(\bar Q \bar P^{-1} \bar Q^\top)
	&= \tr(\bar P^{-1} \bar Q^\top \bar Q) \\
	&\geq \SVMin(\bar P^{-1}) \tr(\bar Q^\top \bar Q) \\
	&= \frac{\|\bar Q\|_\rmF^2}{\|\bar P\|} \\
	&\geq \frac{\|\bar Q\|^2}{\tr(\bar P)}.
\end{align*}
Therefore,
\begin{align*}
	\|\bar K\|^2 + 1
	&\leq \tr(\bar P) + \tr(\bar K \bar P \bar K^\top)
	\leq \frac{\chi}{\eta_2},
\end{align*}
and
\begin{align*}
	\|\bar Q\|
	&\leq \sqrt{\tr(\bar P) \tr(\bar Q \bar P^{-1} \bar Q^\top)} \\
	&\leq \frac{\tr(\bar P) + \omega \tr(\bar Q \bar P^{-1} \bar Q^\top)}{2 \sqrt{\omega}} \\
	&\leq \frac{\chi}{2 \sqrt{\omega} \eta_2}.
\end{align*}
Combining these bounds and using
\[
	\left\| \begin{bmatrix} \bar K \\ I \end{bmatrix} \right\|
	\leq
	\sqrt{\|\bar K\|^2 + 1}
\]
and \eqref{eq:quant-err-mx-bnd-A}, we obtain
\begin{align*}
	&\quad\,\, 
	\|(A + B \bar K) \bar Q^\top \Delta_0^\top 
	+ \Delta_0 \bar Q (A + B \bar K)^\top\| \\
	&\leq 
	2 \left\| \begin{bmatrix} B & A \end{bmatrix} \right\|
	\left\| \begin{bmatrix} \bar K \\ I \end{bmatrix} \right\| 
	\|\bar Q\| \|\Delta_0\| \\
	&\leq 
	2 \overline\sigma_{AB} \|\Delta_0\| \|\bar Q\| 
	\sqrt{\|\bar K\|^2 + 1} \\
	&\leq \sqrt{n T} (\|A\| + 1) 
	\overline\sigma_{AB} 
	\sqrt{\frac{\chi^3}{\omega \eta_2^3}}\, \delta.
\end{align*}

Requiring the above expression to be smaller than $ 1 $ and solving for $ \delta $ yields the following sufficient condition, which is independent of the particular optimal solution.
If \eqref{eq:ddc-feas-cond-1} holds and
\begin{equation}\label{eq:ddc-feas-cond-2}
	\delta 
	< \frac{1}{\sqrt{n T} (\|A\| + 1) \overline\sigma_{AB}} 
	\sqrt{\frac{\omega \eta_2^3}{\chi^3}}
	=: \delta_2,
\end{equation}
then 
\[
	\|(A + B \bar K) \bar Q^\top \Delta_0^\top 
	+ \Delta_0 \bar Q (A + B \bar K)^\top\| 
	< 1.
\]
Since the matrix inside the norm is symmetric,
\[
	-(A + B \bar K) \bar Q^\top \Delta_0^\top 
	- \Delta_0 \bar Q (A + B \bar K)^\top 
	\prec I.
\]
Consequently, $ \bar\Psi \prec I $.
Therefore, $ A + B \bar K $ is Schur stable.

Finally, if
\begin{equation}\label{eq:ddc-feas-cond}
	\delta < \min\{\delta_1,\, \delta_2\} =: \delta^*,
\end{equation}
where $ \delta_1, \delta_2 > 0 $ are defined in \eqref{eq:ddc-feas-cond-1} and \eqref{eq:ddc-feas-cond-2}, respectively, then the conclusions of Theorem~\ref{thm:ddc-feas} hold.
This completes the proof.

Since $ \eta_2 \in (0, 1) $ is arbitrary, the sufficient bound on $ \delta $ can be tightened by taking the supremum of $ \delta^* $ over $ \eta_2 \in (0, 1) $.

%%%%%%%%%%%%%%%%%%%%
\subsubsection{Error bound for verifiable stability condition}\label{sssec:ddc-feas-stab-cond}
In this subsection, we derive a more conservative bound on $ \delta $ that guarantees the verifiable condition \eqref{eq:ddc-stab-cond} for every optimal solution to \eqref{eq:sdp-quant}.

Let $ (\bar\gamma, \bar Q, \bar P, \bar L, \bar W) $ be an arbitrary optimal solution to \eqref{eq:sdp-quant}, and define $ \bar K, \bar M $ as in \eqref{eq:ddc-stab-gain} and $ \chi $ as in \eqref{eq:chi}.
Using 
\[
	A + B \bar K = (\hat X_1 - \Delta_0) \bar Q \bar P^{-1}
\]
and \eqref{eq:quant-err-mx-bnd}, we obtain
\begin{align*}
	&\quad\,\, 
	\hat\delta^2 \|\bar M\| 
	+ 2 \hat\delta \|\hat X_1 \bar M\| \\
	&= 
	\hat\delta^2 \|\bar M\| 
	+ 2 \hat\delta \|(\hat X_1 - \Delta_0) \bar M + \Delta_0 \bar M\| \\
	&\leq 
	\hat\delta^2 \|\bar M\| 
	+ 2 \hat\delta \|\Delta_0\| \|\bar M\| + 2 \hat\delta \|(\hat X_1 - \Delta_0) \bar Q \bar P^{-1} \bar Q^\top\| \\
	&\leq 
	3 \hat\delta^2 \|\bar M\| 
	+ 2 \hat\delta \|(A + B \bar K) \bar Q^\top\| \\
	&= 
	3 \hat\delta^2 \|\bar M\| 
	+ 2 \hat\delta 
	\left\| \begin{bmatrix} B & A \end{bmatrix} \right\|
	\left\| \begin{bmatrix} \bar K \\ I \end{bmatrix} \right\| \|\bar Q \| \\
	&\leq 
	3 \hat\delta^2 \|\bar M\| 
	+ 2 \hat\delta \|\bar Q\| \sqrt{\alpha^2 + \beta^2} \sqrt{\|\bar K\|^2 + 1}.
\end{align*}
The derivation in Section~\ref{sssec:ddc-feas-stab} yields
\begin{align*}
	\|\bar K\|^2 + 1
	\leq \frac{\chi}{\eta_2},
	\qquad
	\|\bar Q\|
	\leq \frac{\chi}{2 \sqrt{\omega} \eta_2}.
\end{align*}
Moreover,
\[
	\|\bar M\|
	\leq \tr(\bar M) 
	= \tr(\bar Q \bar P^{-1} \bar Q^\top) 
	\leq \frac{\bar\gamma}{\omega} 
	\leq \frac{\chi}{\omega \eta_2}.
\]
Therefore,
\begin{align*}
	&\quad\,\, 
	\hat\delta^2 \|\bar M\| 
	+ 2 \hat\delta \|\hat X_1 \bar M\| \\
	&\leq 
	3 \hat\delta^2 \|\bar M\| 
	+ 2 \hat\delta \|\bar Q\| \sqrt{\alpha^2 + \beta^2} \sqrt{\|\bar K\|^2 + 1} \\
	&\leq \frac{3 \hat\delta^2 \chi}{\omega \eta_2} 
	+ \hat\delta \sqrt{\alpha^2 + \beta^2} \sqrt{\frac{\chi^3}{\omega \eta_2^3}} \\
	&= 
	\frac{3 \chi}{\omega \eta_2} 
	\left( \hat\delta + \frac{1}{6} \sqrt{\alpha^2 + \beta^2} \sqrt{\frac{\omega \chi}{\eta_2}} \right)^2 
	- \frac{\chi^2}{12 \eta_2^2} (\alpha^2 + \beta^2).
\end{align*}

Requiring the above expression to be smaller than $ 1 $ and using $ \hat\delta = \sqrt{nT}(\alpha + 1) \delta $ yields the following sufficient condition.
If \eqref{eq:ddc-feas-cond-1} holds and
\begin{equation}\label{eq:ddc-feas-cond-2-alt}
	\delta 
	< 
	\frac{1}{6 \sqrt{n T} (\alpha + 1)} \sqrt{\frac{\omega \chi}{\eta_2}} 
	\left( \sqrt{\alpha^2 + \beta^2 + \frac{12 \eta_2^2}{\chi^2}} - \sqrt{\alpha^2 + \beta^2} \right) 
	=: \bar\delta_2,
\end{equation}
then every optimal solution to \eqref{eq:sdp-quant} satisfies \eqref{eq:ddc-stab-cond}.

This alternative gives the stronger conclusion that every optimal solution satisfies the verifiable stability condition \eqref{eq:ddc-stab-cond}, but at the expense of a more conservative bound on $ \delta $.
Indeed, since
\begin{align*}
	\bar\delta_2 
	&= 
	\frac{2}{\sqrt{n T} (\alpha + 1) 
	\left( \sqrt{\alpha^2 + \beta^2 + \frac{12 \eta_2^2}{\chi^2}} + \sqrt{\alpha^2 + \beta^2} \right)} 
	\sqrt{\frac{\omega \eta_2^3}{\chi^3}} \\
	&< 
	\frac{1}{\sqrt{n T} (\alpha + 1) \sqrt{\alpha^2 + \beta^2}} 
	\sqrt{\frac{\omega \eta_2^3}{\chi^3}} \\
	&< \delta_2,
\end{align*}
\eqref{eq:ddc-feas-cond-2-alt} yields a more conservative bound on $ \delta $ than \eqref{eq:ddc-feas-cond-2}.

Finally, if
\begin{equation}\label{eq:ddc-feas-cond-alt}
	\delta < \min\{\delta_1,\, \bar\delta_2\} =: \bar\delta^*,
\end{equation}
where $ \delta_1, \bar\delta_2 > 0 $ are defined in \eqref{eq:ddc-feas-cond-1} and \eqref{eq:ddc-feas-cond-2-alt}, respectively, then every optimal solution to \eqref{eq:sdp-quant} satisfies \eqref{eq:ddc-stab-cond}.

\begin{rmk}
The proof of Theorem~\ref{thm:ddc-feas} also provides bounds on the optimal solutions to the quantized SDP \eqref{eq:sdp-quant} that are uniform over sufficiently small quantization errors.
In particular, for any fixed $ \eta_2 \in (0, 1) $, if $ \delta \leq \delta_1 $, then every optimal solution $ (\bar\gamma, \bar Q, \bar P, \bar L, \bar W) $ to \eqref{eq:sdp-quant} satisfies
\[ 
	\bar\gamma 
	\leq \frac{\chi}{\eta_2},
	\qquad
	\|\bar Q\| 
	\leq \frac{\chi}{2 \sqrt{\omega} \eta_2},
\]
and $ \bar K, \bar M $ defined in \eqref{eq:ddc-stab-gain} satisfy
\[
	\|\bar K\|^2 + 1 
	\leq \frac{\chi}{\eta_2},
	\qquad
	\|\bar M\|\leq\frac{\chi}{\omega\eta_2},
\]
where $ \delta_1 $ and $ \chi $ depend on $ \eta_2 $ and are defined in \eqref{eq:ddc-feas-cond-1} and \eqref{eq:chi}, respectively.
These bounds are independent of the particular optimal solution.
\end{rmk}

%%%%%%%%%%%%%%%%%%%%%%%%%%%%%%%%%%%%%%%%%%%%%%%%%%%%%%%%%%%%%%%%%%%%%%%%%%%%%%%%
\section{Quantized feedback stabilization}\label{sec:stab}
Having obtained a stabilizing feedback gain from quantized data, we now address closed-loop stabilization using quantized state feedback.

%%%%%%%%%%%%%%%%%%%%%%%%%%%%%%%%%%%%%%%%
\subsection{Quantized state feedback implementation}\label{ssec:stab-quant}
At time $ T $, suppose that a feasible solution $ (\bar\gamma, \bar Q, \bar P, \bar L, \bar W) $ to the quantized SDP \eqref{eq:sdp-quant} has been obtained and satisfies \eqref{eq:ddc-stab-cond}. 
Then by Theorem~\ref{thm:ddc-stab}, the feedback gain $ \bar K := U_0 \bar Q \bar P^{-1} $ satisfies \eqref{eq:ddc-stab}, and $ A + B \bar K $ is Schur stable.

For the transition between the two stages, we let $ u(T) = 0 $ and begin quantized state feedback at time $ T+1 $.
During the stabilization stage, the quantizer uses a fixed number of quantization cells and exponentially decaying bounds:
\begin{equation}\label{eq:quant-num-bnd-stab}
	N_k = N_s,
	\quad
	E_{k+1} = \mu E_k,
	\qquad
	k = T+1, T+2, \ldots,
\end{equation}
where the initial bound $ E_{T+1} $, the contraction factor $ \mu $, and a sufficient condition on $ N_s $ ensuring $ \mu \in (0, 1) $ are derived in Section~\ref{ssec:stab-conv}.
With these choices, the closed-loop state under
\begin{equation}\label{eq:sys-cl}
	x(k+1) = A x(k) + B \bar K \hat x(k),
	\qquad k \geq T+1,
\end{equation}
remains in the quantization range (i.e., \eqref{eq:quant-bnd} holds) for all $ k \geq T + 1 $ and converges exponentially to the origin.

During this stage, since the number of quantization cells remains constant, the quantizer can be expressed as a ``zooming'' quantizer \cite{Liberzon2003Book}:
\[
	\hat x(k) 
	= \mu_k q \left( \frac{x(k)}{\mu_k} \right),
\]
where $ q(\cdot) $ is the quantizer with bound $ E_{T+1} $ and $ N_s $ quantization cells per dimension, and
\[
	\mu_k := \mu^{k - T - 1}, 
	\qquad k = T+1, T+2, \ldots
\]
is the zoom variable.

%%%%%%%%%%%%%%%%%%%%%%%%%%%%%%%%%%%%%%%%
\subsection{Exponential convergence}\label{ssec:stab-conv}
We establish exponential convergence by constructing nested sets that contract under the closed-loop dynamics, based on sublevel sets of
\begin{equation}\label{eq:lya-fcn}
	V(x) := x^\top \bar P^{-1} x,
\end{equation}
where $ \bar P \succ 0 $ is from the feasible solution to \eqref{eq:sdp-quant} obtained at time $ T $. 

We first construct a sublevel set of $ V $ containing $ x(T+1) $ and derive $ E_{T+1} $ such that the quantization range $ \calR(E_{T+1}) $ contains this set.
Recall that
\[
	\calR(E) = \{x \in \R^n: \|x\|_\infty \leq E\},
\]
and define the sublevel set
\[
	\calS(E) := \left\{ x \in \R^n: V(x) \leq \frac{E^2}{\eigMax(\bar P)} \right\}.
\]
Since
\[
	\frac{\|x\|_\infty^2}{\eigMax(\bar P)} 
	\leq \frac{\|x\|^2}{\eigMax(\bar P)}
	\leq V(x) 
	\leq \frac{\|x\|^2}{\eigMin(\bar P)} 
	\leq \frac{n \|x\|_\infty^2}{\eigMin(\bar P)}
	\qquad \forall x \in \R^n,
\]
we have
\[ 
	\calS(E) \subset \calR(E) 
	\subset \calS \left( \sqrt{\frac{n \eigMax(\bar P)}{\eigMin(\bar P)}} E \right)
	\qquad \forall E > 0.
\]
Using $ u(T) = 0 $ and $ x(T) \in \calR(E_T) $, we obtain
\begin{align*}
	\|x(T+1)\|_\infty 
	\leq \|A x(T)\| 
	\leq \alpha \|x(T)\|
	\leq \alpha \sqrt{n} E_T.
\end{align*}
Therefore, defining
\begin{equation}\label{eq:quant-bnd-stab-ini}
	E_{T+1} := \alpha n \sqrt{\frac{\eigMax(\bar P)}{\eigMin(\bar P)}} E_T
\end{equation}
yields
\[
	x(T+1) 
	\in \calR(\alpha \sqrt{n} E_T) 
	\subset \calS(E_{T+1}) 
	\subset \calR(E_{T+1}).
\]
Thus, $ \calR(E_{T+1}) $ can be used as the initial quantization range for the stabilization stage.

A tighter initial bound can be obtained by exploiting the available quantized measurement $ \hat x(T) $.
By \eqref{eq:quant-err-bnd-ddc},
\[
	\|x(T)\|
	\leq \|\hat x(T)\| + \|d(T)\|
	\leq \|\hat x(T)\| + \sqrt{n} \delta.
\]
Therefore, letting
\begin{equation}\label{eq:quant-bnd-stab-ini-data}
	E_{T+1} 
	:= 
	\alpha
	\sqrt{\frac{n \eigMax(\bar P)}{\eigMin(\bar P)}}
	(\|\hat x(T)\| + \sqrt{n} \delta)
\end{equation}
also ensures that
\[
	x(T+1) 
	\in \calR(\alpha \|x(T)\|) 
	\subset \calS(E_{T+1}) 
	\subset \calR(E_{T+1}).
\]

Our third theorem establishes a sufficient condition for exponential convergence.
For convenience, define
\begin{equation}\label{eq:stab-conv-coef}
	\theta := 
	\frac{
		(\alpha + \beta \|\bar K\|) \beta \|\bar K\| \eigMax(\bar P)
	}{
		\eigMin(\bar P)
	}, 
	\qquad
	\phi := 
	\frac{
		\beta^2 \|\bar K\|^2
	}{
		\eigMin(\bar P)
	},
\end{equation}
where $ \alpha, \beta > 0 $ are as in \eqref{eq:sys-mx-bnd}.

\begin{thm}\label{thm:stab-conv}
Suppose that the number of quantization cells satisfies
\begin{equation}\label{eq:stab-conv-cond}
	N_s 
	> \frac{
		\theta + \sqrt{\theta^2 + \eta_1 \phi}
	}{
		\eta_1
	}
	\sqrt{\frac{
		n \eigMax(\bar P)^3
	}{
		\eigMin(\bar P)
	}} 
	=: N^*,
\end{equation}
where $ \eta_1 \in (0, 1) $ is defined in \eqref{eq:ddc-stab-coef}.
Then there exists $ \varepsilon > 0 $ such that
\begin{equation}\label{eq:stab-conv-cond-eps}
	N^*_\varepsilon
	:= 
	\sqrt{\frac{
		\bigl( \theta + \sqrt{\theta^2 + \eta_1 (\phi + \varepsilon)} \bigr)^2 n \eigMax(\bar P)^3
	}{
		\eta_1^2 \eigMin(\bar P)
	} 
	- \varepsilon n \eigMax(\bar P)} 
\end{equation}
satisfies $ N_s > N^*_\varepsilon > N^* $.
Moreover, for any such $ \varepsilon $,
\begin{equation}\label{eq:stab-conv}
	\|x(k)\|_\infty 
	\leq \mu^{k - T - 1} E_{T+1} 
	\qquad \forall k \geq T + 1,
\end{equation}
where
\begin{equation}\label{eq:stab-conv-rate}
	\mu 
	:= \sqrt{\max \left\{ 
		\frac{(N^*_\varepsilon)^2}{N_s^2},\, 
		1 - \frac{n \varepsilon \eigMax(\bar P)}{N_s^2} 
	\right\}} < 1.
\end{equation}
\end{thm}

The proof of Theorem~\ref{thm:stab-conv} relies on the following lemma.
\begin{lem}\label{lem:stab-lya}
For any $ k \geq T + 1 $, if
\[
	x(k) \in \calS(E_k),
\]
then
\begin{equation*}
	x(k+1) \in \calS(\mu E_k),
\end{equation*}
where $ \mu $ is defined in \eqref{eq:stab-conv-rate}.
\end{lem}
\begin{proof}
For brevity, let $ x = x(k) $, $ x^+ = x(k+1) $, $ V = V(x) $, $ V^+ = V(x^+) $, $ d = d(k) $, and $ E = E_k $ in this proof.
The closed-loop dynamics \eqref{eq:sys-cl} gives
\begin{align*}
	V^+ &= 
	\bigl( (A + B \bar K) x + B \bar K d \bigr)^\top 
	\bar P^{-1} 
	\bigl( (A + B \bar K) x + B \bar K d \bigr) \\
	&= 
	x^\top (A + B \bar K)^\top \bar P^{-1} (A + B \bar K) x \\
	&\quad\, + 2 x^\top (A + B \bar K)^\top \bar P^{-1} B \bar K d + d^\top \bar K^\top B^\top \bar P^{-1} B \bar K d.
\end{align*}

We first bound the nominal closed-loop term.
By \eqref{eq:ddc-stab},
\[
	(A + B \bar K) \bar P (A + B \bar K)^\top 
	\preceq \bar P - \eta_1 I 
	\preceq \left( 1 - \frac{\eta_1}{\eigMax(\bar P)} \right) \bar P.
\]
Thus,
\[
	\bar P^{-1/2} (A + B \bar K) 
	\bar P 
	(A + B \bar K)^\top \bar P^{-1/2} 
	\preceq 
	\left( 1 - \frac{\eta_1}{\eigMax(\bar P)} \right) I,
\]
which implies
\[
	\left\| \bar P^{-1/2} (A + B \bar K) \bar P^{1/2} \right\|^2 
	\leq 1 - \frac{\eta_1}{\eigMax(\bar P)}.
\]
Since a matrix and its transpose have the same spectral norm, we have
\[
	\bar P^{1/2} (A + B \bar K)^\top
	\bar P^{-1} 
	(A + B \bar K) \bar P^{1/2} 
	\preceq 
	\left( 1 - \frac{\eta_1}{\eigMax(\bar P)} \right) I,
\]
and hence
\[
	(A + B \bar K)^\top
	\bar P^{-1} 
	(A + B \bar K)
	\preceq 
	\left( 1 - \frac{\eta_1}{\eigMax(\bar P)} \right) \bar P^{-1}.
\]
Consequently,
\begin{align*}
	x^\top (A + B \bar K)^\top \bar P^{-1} (A + B \bar K) x
	&\leq 
	\left( 1 - \frac{\eta_1}{\eigMax(\bar P)} \right) x^\top \bar P^{-1} x \\
	&\leq 
	V - \frac{\eta_1}{\eigMax(\bar P)^2} \|x\|^2.
\end{align*}
We next bound the terms due to quantization.
Since $ x \in \calS(E) \subset \calR(E) $, by \eqref{eq:sys-mx-bnd} and \eqref{eq:quant-err-bnd},
\begin{align*}
	&\quad\,\, x^\top (A + B \bar K)^\top \bar P^{-1} B \bar K d \\
	&\leq 
	\frac{
		(\alpha + \beta \|\bar K\|) \beta \|\bar K\|
	}{
		\eigMin(\bar P)
	} 
	\|x\| \sqrt{n} \|d\|_\infty \\
	&\leq 
	\frac{
		(\alpha + \beta \|\bar K\|) \beta \|\bar K\| \eigMax(\bar P)
	}{
		\eigMin(\bar P)
	} 
	\frac{\sqrt{n} E}{\sqrt{\eta_1} N_s} 
	\frac{\sqrt{\eta_1}}{\eigMax(\bar P)} 
	\|x\| \\
	&= 
	\frac{\theta \sqrt{n} E}{\sqrt{\eta_1} N_s} 
	\frac{\sqrt{\eta_1}}{\eigMax(\bar P)} 
	\|x\|,
\end{align*}
where $ \theta $ is defined in \eqref{eq:stab-conv-coef}, and
\begin{align*}
	d^\top \bar K^\top B^\top \bar P^{-1} B \bar K d
	&\leq 
	\frac{\beta^2 \|\bar K\|^2}{\eigMin(\bar P)} 
	n \|d\|_\infty^2 \\
	&\leq 
	\frac{\beta^2 \|\bar K\|^2}{\eigMin(\bar P)} \frac{n E^2}{N_s^2} \\
	&\leq 
	\frac{\phi n E^2}{N_s^2},
\end{align*}
where $ \phi $ is defined in \eqref{eq:stab-conv-coef}.
Combining these bounds and completing the square yields
\begin{equation}\label{eq:stab-lya-bnd}
\begin{aligned}
	V^+
	&\leq V 
	- \frac{\eta_1}{\eigMax(\bar P)^2} \|x\|^2 
	+ \frac{2 \theta \sqrt{n} E}{\sqrt{\eta_1} N_s} 
	\frac{\sqrt{\eta_1}}{\eigMax(\bar P)} \|x\| 
	+ \frac{\phi n E^2}{N_s^2} \\
	&= V 
	- \left( 
		\frac{\sqrt{\eta_1}}{\eigMax(\bar P)} \|x\| 
		- \frac{\theta \sqrt{n} E}{\sqrt{\eta_1} N_s} 
	\right)^2 
	+ \left(
		\frac{\theta^2}{\eta_1} + \phi
	\right) 
	\frac{n E^2}{N_s^2}.
\end{aligned}
\end{equation}

We consider two cases according to the magnitude of the state.
If
\begin{equation}\label{eq:stab-lya-case}
	\|x\| 
	\geq \frac{\eigMax(\bar P)}{\sqrt{\eta_1}} 
	\left(
		\frac{\theta}{\sqrt{\eta_1}} + \sqrt{\frac{\theta^2}{\eta_1} + \phi + \varepsilon}
	\right)
	\frac{\sqrt{n} E}{N_s},
\end{equation}
then the squared term in \eqref{eq:stab-lya-bnd} is at least $ (\theta^2/\eta_1 + \phi + \varepsilon) n E^2/N_s^2 $.
Since $ x \in \calS(E) $, we have
\[
	V^+
	\leq 
	V - \frac{\varepsilon n E^2}{N_s^2} 
	\leq
	\left( 1 - \frac{\varepsilon n \eigMax(\bar P)}{N_s^2} \right) 
	\frac{E^2}{\eigMax(\bar P)}.
\]
Otherwise, using $ V \leq \|x\|^2/\eigMin(\bar P) $ in \eqref{eq:stab-lya-bnd} gives
\begin{align*}
	V^+ 
	&\leq 
	\left( 
		\frac{1}{\eigMin(\bar P)} 
		- \frac{\eta_1}{\eigMax(\bar P)^2} 
	\right) \|x\|^2 
	+ \frac{2 \theta \sqrt{n} E}{N_s \eigMax(\bar P)} \|x\| 
	+ \frac{\phi n E^2}{N_s^2}.
\end{align*}
Since $ \eigMax(\bar P) \geq \eigMin(\bar P) \geq 1 > \eta_1 $, the right-hand side is nondecreasing in $ \|x\| $.
Using the upper bound on $ \|x\| $ given by the complement of \eqref{eq:stab-lya-case}, we obtain
\begin{align*}
	V^+
	&< 
	\frac{\eigMax(\bar P)^2}{\eta_1 \eigMin(\bar P)} 
	\left(
		\frac{\theta}{\sqrt{\eta_1}} + \sqrt{\frac{\theta^2}{\eta_1} + \phi + \varepsilon}
	\right)^2
	\frac{n E^2}{N_s^2} 
	- \frac{\varepsilon n E^2}{N_s^2} \\
	&= 
	\left( 
		\frac{\big( \theta + \sqrt{\theta^2 + \eta_1 (\phi + \varepsilon)} \bigr)^2  \eigMax(\bar P)^2}
		{\eta_1^2 \eigMin(\bar P)} 
	- \varepsilon \right) 
	\frac{n E^2}{N_s^2} \\
	&= 
	\frac{(N^*_\varepsilon)^2}{N_s^2} 
	\frac{E^2}{\eigMax(\bar P)},
\end{align*}
where $ N^*_\varepsilon $ is defined in \eqref{eq:stab-conv-cond-eps}.
Combining both cases, we obtain
\[
	V^+ \leq \frac{\mu^2 E^2}{\eigMax(\bar P)},
\]
where $ \mu $ is defined in \eqref{eq:stab-conv-rate}. 
Hence $ x^+ \in \calS(\mu E) $.
\end{proof}

\begin{proof}[Proof of Theorem~\ref{thm:stab-conv}]
First, since 
$ \bigl( \theta + \sqrt{\theta^2 + \eta_1 (\phi + \varepsilon)} \bigr)^2 \geq \eta_1 \varepsilon $ 
and $ \eigMax(\bar P)^2 > \eta_1 \eigMin(\bar P) $, 
it follows that $ N^*_\varepsilon $ is well defined and positive for every $ \varepsilon > 0 $.
Moreover, $ N^*_\varepsilon $ is strictly increasing in $ \varepsilon $, and
\[
    \lim_{\varepsilon \to 0^+} N^*_\varepsilon = N^*.
\]
Thus, under \eqref{eq:stab-conv-cond}, there exists a sufficiently small $ \varepsilon > 0 $ such that $ N_s > N^*_\varepsilon > N^* $. 
For any such $ \varepsilon $, both terms in the maximum in \eqref{eq:stab-conv-rate} are strictly smaller than one, and $ N^*_\varepsilon > 0 $. Hence $ \mu \in (0, 1) $.

By either \eqref{eq:quant-bnd-stab-ini} or \eqref{eq:quant-bnd-stab-ini-data},
\[
	x(T+1) 
	\in \calS(E_{T+1}).
\]
Therefore, iterating Lemma~\ref{lem:stab-lya} gives
\[
	x(k) 
	\in 
	\calS(\mu^{k-T-1} E_{T+1}) 
	\subset 
	\calR(\mu^{k-T-1} E_{T+1})
	\qquad 
	\forall k \geq T + 1,
\]
which proves \eqref{eq:stab-conv}.
\end{proof}

\begin{rmk}
Note that the first term in the maximum in \eqref{eq:stab-conv-rate} is strictly increasing in $ \varepsilon $, whereas the second term in the maximum is strictly decreasing. 
Consequently, the smallest bound on the contraction factor $ \mu $ is obtained by choosing $ \varepsilon $ such that the two terms are equal.	
\end{rmk}

%%%%%%%%%%%%%%%%%%%%%%%%%%%%%%%%%%%%%%%%%%%%%%%%%%%%%%%%%%%%%%%%%%%%%%%%%%%%%%%%
\section{Simulation results}\label{sec:sim}

%%%%%%%%%%%%%%%%%%%%%%%%%%%%%%%%%%%%%%%%
%\subsection{Quantized controller design and stabilization}
We illustrate the proposed approach on the LTI system \eqref{eq:sys} with
\[
	A = \begin{bmatrix}
		2 & 1 \\
		-1 & -1
	\end{bmatrix}, \qquad
	B = \begin{bmatrix}
		1 \\
		0
	\end{bmatrix},
\]
which is unstable in open loop.
The dimensions are $ n = 2 $ and $ m = 1 $.
For the controller design stage, we use $ T = m (n+1) + n = 5 $, the minimum data length allowed by \eqref{eq:ddc-T-cond}.
The stabilization stage is simulated for $ 30 $ steps.

During the design stage, the input sequence is PE of order $ n+1 = 3 $ and satisfies $ \|u(k)\| \leq \bar u = 1 $.
The initial state is $ x_0 = \begin{bmatrix} 1 & 1 \end{bmatrix}^\top $.
We use $ E_0 = 1 $, $ \delta = 0.01 $, $ \alpha = 3 > \|A\| $, and $ \beta = 1.5 > \|B\| $.
Using \eqref{eq:quant-lvl-ddc} together with the measurement-dependent update rule \eqref{eq:quant-bnd-ddc-data}, the numbers of quantization cells per dimension are
\[ 
	(N_0, N_1, \ldots, N_5) = (100, 461, 1174, 951, 1742, 2658).
\]

The collected data satisfy the rank condition \eqref{eq:rank-cond-quant}.
Solving the quantized SDP \eqref{eq:sdp-quant} with $ \omega = 1 $ yields 
\[ 
	\bar K 
	= U_0 \bar Q \bar P^{-1} 
	= \begin{bmatrix} -1.129 & -0.333 \end{bmatrix}, 
\]
and $ \eta_1 = 0.338 > 0 $.
Thus, the verifiable condition \eqref{eq:ddc-stab-cond} holds, and Theorem~\ref{thm:ddc-stab} guarantees that $ A + B \bar K $ is Schur stable. Indeed, its eigenvalues are $ 0.392 $ and $ -0.521 $.

During the stabilization stage, the initial quantization range is computed using the measurement-dependent initialization \eqref{eq:quant-bnd-stab-ini-data}, yielding $ E_{T+1} = 118.9 $.
The sufficient condition \eqref{eq:stab-conv-cond} in Theorem~\ref{thm:stab-conv} gives the bound $ N^* = 2623.6 $ and hence $ N_s \geq 2624 $.
In the simulation, we use $ N_s = 50 $ cells per dimension and set $ \mu = 0.95 $. 
Although these parameter choices are not certified by Theorem~\ref{thm:stab-conv}, the simulation results in Fig.~\ref{fig:sim} show successful convergence to the origin.

\begin{figure}[!ht]
\centering
\begin{subfigure}{.475\textwidth}
\centering
\includegraphics[width=1\textwidth,trim={10pt 6pt 10pt 0pt}]{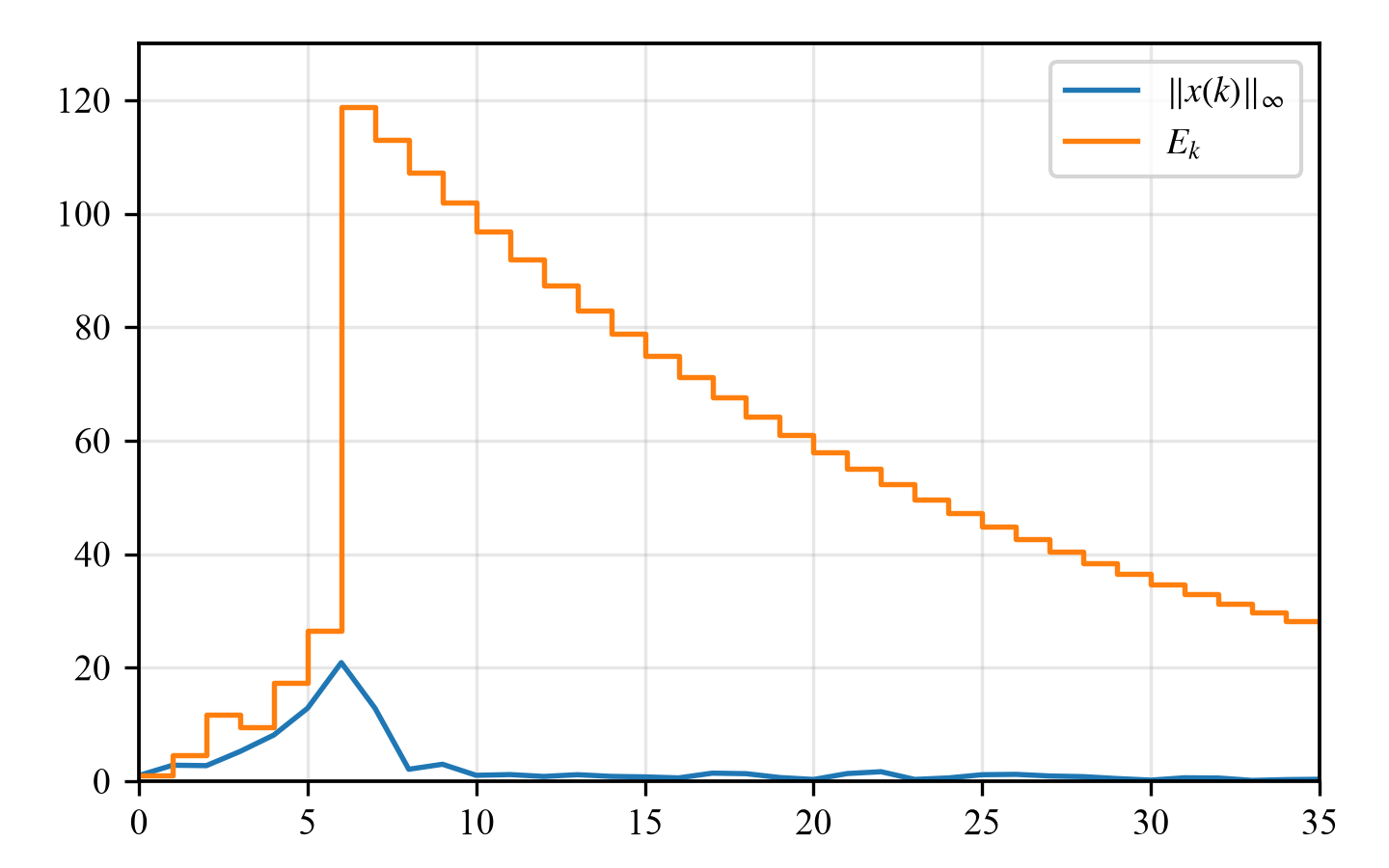}% ,frame,clip
\caption{}
\label{fig:sim-E}
\end{subfigure}%
\hfill
% \\
\begin{subfigure}{.475\textwidth}
\centering
\includegraphics[width=1\textwidth,trim={10pt 6pt 10pt 0pt}]{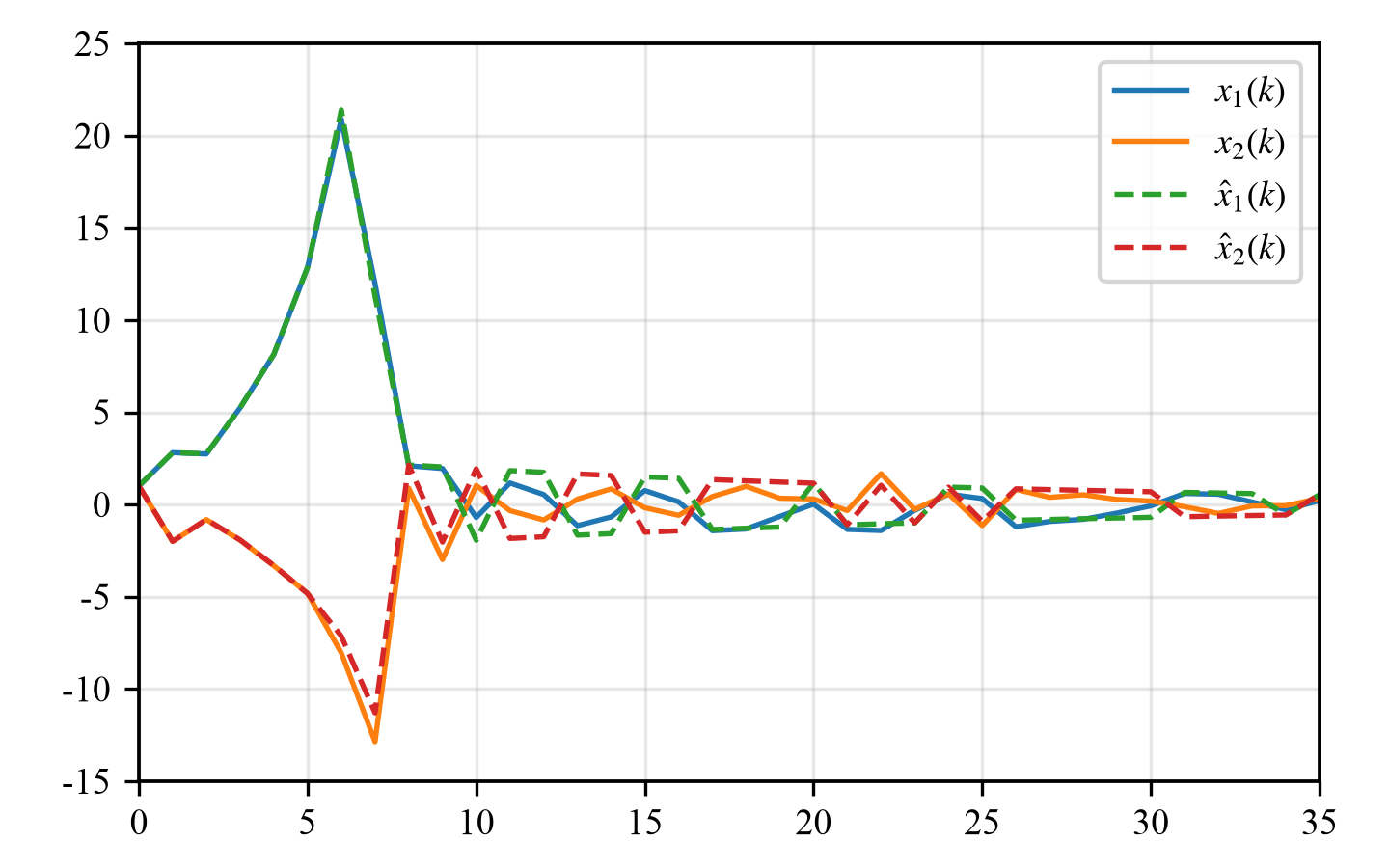}% 
\caption{}
\label{fig:sim-x}
\end{subfigure}%
\caption{Evolution of
(\subref{fig:sim-E}) the state norm $ \|x(k)\|_\infty $ and quantization bound $ E_k $, and
(\subref{fig:sim-x}) the state components $ x_i(k) $ and quantized measurements $ \hat x_i(k) $.
The design stage ends at $ k = 5 $, and quantized state feedback begins at $ k = 6 $.}
\label{fig:sim}
\end{figure}
Figure~\ref{fig:sim}(\subref{fig:sim-E}) shows that the state remains within the quantization range and converges to the origin.
During the design stage ($ k = 0, \ldots, 5 $), the quantizer maintains the prescribed quantization error bound.
At the transition to stabilization, the quantization range increases to $ E_{T+1} = 118.9 $ according to \eqref{eq:quant-bnd-stab-ini-data}, and thereafter decreases exponentially with $ \mu = 0.95 $.

Figure~\ref{fig:sim}(\subref{fig:sim-x}) shows that the quantized measurements closely track the states during data collection, and that both the states and quantized measurements converge to zero during stabilization.

The sufficient bound in Theorem~\ref{thm:stab-conv} is conservative in this example. 
Although it requires $ N_s \geq 2624 $, convergence is achieved with $ N_s = 50 $ in simulation.
This conservativeness is due partly to the worst-case norm bounds in the analysis and the conversion between the hypercubic quantization ranges and the Lyapunov sublevel sets.
More careful analysis may substantially reduce the quantization requirement.

The trade-off between the quantization requirements for controller design and feedback stabilization is investigated in \cite{LiYang2026Allerton}.

%%%%%%%%%%%%%%%%%%%%%%%%%%%%%%%%%%%%%%%%%%%%%%%%%%%%%%%%%%%%%%%%%%%%%%%%%%%%%%%%
\section{Conclusion}\label{sec:end}
This paper developed a data-driven framework for stabilizing unknown discrete-time linear systems using only quantized state measurements that take values in finite sets. 
In the controller design stage, quantized trajectory data are used to construct a feedback gain together with a Lyapunov matrix, with sufficient conditions established to guarantee closed-loop stability. 
In the stabilization stage, a Lyapunov-based quantizer with a fixed finite number of quantization cells ensures exponential convergence of the closed-loop system.
Thus, both controller design and feedback stabilization are achieved under finite data rates.

Future research directions include reducing the conservativeness of the sufficient quantization requirements, 
removing the reliance on known bounds on $ \|A\| $ and $ \|B\| $,
and extending the proposed framework to continuous-time LTI systems and switched systems.

%%%%%%%%%%%%%%%%%%%%%%%%%%%%%%%%%%%%%%%%%%%%%%%%%%%%%%%%%%%%%%%%%%%%%%%%%%%%%%%%
%% Reference

%\bstctlcite{IEEEexample:BSTcontrol}
\footnotesize
\addcontentsline{toc}{section}{References}
\bibliography{reference,reference-temp}

@article{PersisTesi2020,
author = {{De Persis}, Claudio and Tesi, Pietro},
journal = {IEEE Transactions on Automatic Control},
month = mar,
number = {3},
pages = {909--924},
title = {{Formulas for data-driven control: Stabilization, optimality, and robustness}},
volume = {65},
year = {2020}
}

@inproceedings{CoulsonLygerosDorfler2019,
address = {Naples, Italy},
author = {Coulson, Jeremy and Lygeros, John and D{\"{o}}rfler, Florian},
booktitle = {2019 European Control Conference},
pages = {307--312},
title = {{Data-enabled predictive control: In the shallows of the DeePC}},
year = {2019}
}

@article{WaardeCamlibelMesbahi2022,
author = {van Waarde, Henk J. and Camlibel, M. Kanat and Mesbahi, Mehran},
journal = {IEEE Transactions on Automatic Control},
month = jan,
number = {1},
pages = {162--175},
title = {{From noisy data to feedback controllers: Nonconservative design via a matrix S-lemma}},
volume = {67},
year = {2022}
}

@article{MartinSchonAllgower2023,
author = {Martin, Tim and Sch{\"{o}}n, Thomas B. and Allg{\"{o}}wer, Frank},
journal = {Annual Reviews in Control},
pages = {100911},
title = {{Guarantees for data-driven control of nonlinear systems using semidefinite programming: A survey}},
volume = {56},
year = {2023}
}

@book{HornJohnson1991,
address = {Cambridge, UK},
author = {Horn, Roger A. and Johnson, Charles R.},
publisher = {Cambridge University Press},
title = {{Topics in Matrix Analysis}},
year = {1991}
}

@article{JiangLiu2013,
author = {Jiang, Zhong-Ping and Liu, Tengfei},
journal = {Acta Automatica Sinica},
month = nov,
number = {11},
pages = {1820--1830},
title = {{Quantized nonlinear control --- A survey}},
volume = {39},
year = {2013}
}

@article{MarkovskyDorfler2021,
author = {Markovsky, Ivan and D{\"{o}}rfler, Florian},
journal = {Annual Reviews in Control},
pages = {42--64},
title = {{Behavioral systems theory in data-driven analysis, signal processing, and control}},
volume = {52},
year = {2021}
}

@article{WillemsRapisardaMarkovskyMoor2005,
author = {Willems, Jan C. and Rapisarda, Paolo and Markovsky, Ivan and {De Moor}, Bart L.M.},
journal = {Systems \& Control Letters},
month = apr,
number = {4},
pages = {325--329},
title = {{A note on persistency of excitation}},
volume = {54},
year = {2005}
}

@article{Fu2024,
author = {Fu, Minyue},
journal = {IEEE/CAA Journal of Automatica Sinica},
month = jan,
number = {1},
pages = {5--17},
title = {{A tutorial on quantized feedback control}},
volume = {11},
year = {2024}
}

@article{RotuloPersisTesi2022,
author = {Rotulo, Monica and {De Persis}, Claudio and Tesi, Pietro},
journal = {Automatica},
month = nov,
pages = {110519},
title = {{Online learning of data-driven controllers for unknown switched linear systems}},
volume = {145},
year = {2022}
}

@article{ChangHuangWangLiu2020,
author = {Chang, Xiao-Heng and Huang, Rui and Wang, Huanqing and Liu, Liang},
journal = {IEEE Transactions on Circuits and Systems II: Express Briefs},
month = apr,
number = {4},
pages = {730--734},
title = {{Robust design strategy of quantized feedback control}},
volume = {67},
year = {2020}
}

@article{XuSunWeDuan2024,
author = {Xu, Tao and Sun, Zhiyong and Wen, Guanghui and Duan, Zhisheng},
journal = {IEEE Transactions on Automatic Control},
month = dec,
number = {12},
pages = {8804--8811},
title = {{Data-driven dynamic event-triggered control}},
volume = {69},
year = {2024}
}

@article{WaardeEisingTrentelmanCamlibel2020,
author = {van Waarde, Henk J. and Eising, Jaap and Trentelman, Harry L. and Camlibel, M. Kanat},
journal = {IEEE Transactions on Automatic Control},
month = nov,
number = {11},
pages = {4753--4768},
title = {{Data informativity: A new perspective on data-driven analysis and control}},
volume = {65},
year = {2020}
}

@article{CoulsonWaardeLygerosDorfler2023,
author = {Coulson, Jeremy and van Waarde, Henk J. and Lygeros, John and D{\"{o}}rfler, Florian},
journal = {IEEE Control Systems Letters},
pages = {1243--1248},
title = {{A quantitative notion of persistency of excitation and the robust fundamental lemma}},
volume = {7},
year = {2023}
}

@article{PersisTesi2021,
author = {{De Persis}, Claudio and Tesi, Pietro},
journal = {Automatica},
month = jun,
pages = {109548},
title = {{Low-complexity learning of linear quadratic regulators from noisy data}},
volume = {128},
year = {2021}
}

@article{YangLiberzon2018,
author = {Yang, Guosong and Liberzon, Daniel},
journal = {IEEE Transactions on Automatic Control},
month = jul,
number = {7},
pages = {2107--2122},
title = {{Feedback stabilization of switched linear systems with unknown disturbances under data-rate constraints}},
volume = {63},
year = {2018}
}

@incollection{Astrom1991,
address = {Berlin, Heidelberg},
author = {{\AA}str{\"{o}}m, Karl J.},
booktitle = {Mathematical System Theory},
editor = {Antoulas, Athanasios C.},
pages = {437--450},
publisher = {Springer Berlin Heidelberg},
title = {{Adaptive Control}},
year = {1991}
}

@article{LiZhaoYou2024,
author = {Li, Xingchen and Zhao, Feiran and You, Keyou},
journal = {Journal of Systems Science and Complexity},
month = feb,
number = {1},
pages = {152--168},
title = {{Data-driven control of linear systems via quantized feedback}},
volume = {37},
year = {2024}
}

@book{Liberzon2003Book,
address = {Boston, MA, USA},
author = {Liberzon, Daniel},
publisher = {Birkh{\"{a}}user},
title = {{Switching in Systems and Control}},
year = {2003}
}

@article{Liberzon2014,
author = {Liberzon, Daniel},
journal = {Automatica},
month = feb,
number = {2},
pages = {409--420},
title = {{Finite data-rate feedback stabilization of switched and hybrid linear systems}},
volume = {50},
year = {2014}
}

@article{SongChang2020,
author = {Song, Jia-Sheng and Chang, Xiao-Heng},
journal = {Applied Mathematics and Computation},
month = jul,
pages = {125070},
title = {{$H_\infty$ controller design of networked control systems with a new quantization structure}},
volume = {376},
year = {2020}
}

@article{TakakCetinkayaIshii2025,
author = {Takaki, Iori and Cetinkaya, Ahmet and Ishii, Hideaki},
journal = {IFAC-PapersOnLine},
number = {4},
pages = {103--108},
title = {{Trade-off in quantization between data-driven design and control inputs}},
volume = {59},
year = {2025}
}

@article{ZieglerNichols1942,
author = {Ziegler, John G. and Nichols, Nathaniel B.},
journal = {Transactions of the American Society of Mechanical Engineers},
month = nov,
number = {8},
pages = {759--765},
title = {{Optimum settings for automatic controllers}},
volume = {64},
year = {1942}
}

@inproceedings{TangDaoutidis2022,
address = {Atlanta, GA, USA},
author = {Tang, Wentao and Daoutidis, Prodromos},
booktitle = {2022 American Control Conference},
pages = {1048--1064},
title = {{Data-driven control: overview and perspectives}},
year = {2022}
}

@article{ZamaniYang2026,
author = {Zamani, Mahmoud and Yang, Guosong},
journal = {IEEE Control Systems Letters},
pages = {1501--1506},
title = {{Input-to-state stabilization of linear systems under data-rate constraints}},
volume = {10},
year = {2026}
}

@article{LiuJiang2019,
author = {Liu, Tengfei and Jiang, Zhong-Ping},
journal = {IEEE Transactions on Automatic Control},
month = feb,
number = {2},
pages = {797--803},
title = {{Event-triggered control of nonlinear systems with state quantization}},
volume = {64},
year = {2019}
}

@inproceedings{MillerZhengSznaierHixenbaugh2024,
address = {Toronto, Canada},
author = {Miller, Jared and Zheng, Jian and Sznaier, Mario and Hixenbaugh, Chris},
booktitle = {2024 American Control Conference},
month = jul,
pages = {1146--1151},
title = {{Data-driven superstabilization of linear systems under quantization}},
year = {2024}
}

@inproceedings{LiYang2026Allerton,
address = {Urbana, IL, USA},
note = {to be published},
author = {Li, Zhenghao and Yang, Guosong},
booktitle = {62nd Annual Allerton Conference on Communication, Control, and Computing},
month = sep,
title = {{On data-driven control of linear systems with quantized data}},
year = {2026}
}

@article{HouWang2013,
author = {Hou, Zhong-Sheng and Wang, Zhuo},
journal = {Information Sciences},
month = jun,
pages = {3--35},
title = {{From model-based control to data-driven control: Survey, classification and perspective}},
volume = {235},
year = {2013}
}

@article{NairFagnaniZampieriEvans2007,
author = {Nair, Girish N. and Fagnani, Fabio and Zampieri, Sandro and Evans, Robin J.},
journal = {Proceedings of the IEEE},
month = jan,
number = {1},
pages = {108--137},
title = {{Feedback control under data rate constraints: An overview}},
volume = {95},
year = {2007}
}

\end{document}